\documentclass[final,onefignum,onetabnum]{siamart220329}

\usepackage{lipsum}
\usepackage{amsfonts}
\usepackage{amssymb}
\usepackage{graphicx}
\usepackage{epstopdf}
\usepackage[caption=false]{subfig}
\usepackage{algorithmic}
\ifpdf
  \DeclareGraphicsExtensions{.eps,.pdf,.png,.jpg}
\else
  \DeclareGraphicsExtensions{.eps}
\fi

\newcommand{\XX}{\bf{X}}
\newcommand{\xx}{\bf{x}}
\newcommand{\Pe}{\operatorname{Pe}}
\newcommand{\dd}{\text{d}}
\newcommand{\R}{\mathbb{R}}
\newcommand{\T}{\mathbb{T}}
\newcommand{\ii}{\text{i}}
\newcommand{\ini}{\text{in}}

\newcommand{\Id}{\text{Id}}
\renewcommand{\Re}{\text{Re}}
\newcommand{\e}{\text{e}}
\renewcommand{\bf}[1]{\mathbf{#1}}

\newsiamremark{remark}{Remark}
\newsiamremark{hypothesis}{Hypothesis}
\crefname{hypothesis}{Hypothesis}{Hypotheses}
\newsiamthm{claim}{Claim}

\title{Lane formation and aggregation spots in a model for ants}
\author{Maria Bruna, Martin Burger, \and Oscar de Wit}
\begin{document}
\maketitle

% REQUIRED
\begin{abstract}
We investigate an interacting particle model to simulate a foraging colony of ants, where each ant is represented as an active Brownian particle. The interactions among ants are mediated through chemotaxis, aligning their orientations with the upward gradient of the pheromone field. Unlike conventional models, our study introduces a parameter that enables the reproduction of two distinctive behaviors: the well-known Keller--Segel aggregation into spots and the formation of traveling clusters, without relying on external constraints such as food sources or nests. We consider the associated mean-field limit partial differential equation (PDE) of this system and establish the analytical and numerical foundations for understanding these particle behaviors. Remarkably, the mean-field PDE not only supports aggregation spots and lane formation but also unveils a bistable region where these two behaviors compete. The patterns associated with these phenomena are elucidated by the shape of the growing eigenfunctions derived from linear stability analysis. This study not only contributes to our understanding of complex ant colony dynamics but also introduces a novel parameter-dependent perspective on pattern formation in collective systems.
\end{abstract}

% REQUIRED
\begin{keywords}
active particles, chemotaxis, lane formation, stability analysis
\end{keywords}

% REQUIRED
\begin{MSCcodes}
35Q84, 35R09, 35B35, 35B36, 35B40, 60J70, 92D50
\end{MSCcodes}

\section{Introduction}
\label{sec:intro}
%%%%%%%%%%%%%%%%%%%%%%%%%%%%%%%%%%%%%%%%%%%%%%%%%%%%%%%%%%%%
Collective behaviors of animals can persist over large distances and for long times. One example can be found in the self-organization of ants. Amongst the different known species of ants (Formicidae), ranging in the 20\,000s, there are many ways in which they use chemicals to organize collectively. Using chemical (olfactory) cues, they organize themselves into trails connecting their nests and food sources or into ant armies to capture moving prey \cite{attygalle1985ant,holldobler1995chemistry}. The foraging trails can last for months and extend over vast distances. It is common to see ants move bidirectionally along the trails, split into incoming and outgoing groups, either bringing food back to the nest or going out to get new food \cite{fourcassie2010ant,knaden2016sensory}.

Ant colonies have long amazed researchers with their ability to regulate traffic and prevent jams in crowded conditions, as may occur in trails along confined spaces. In fact, ants appear to do better than humans at traffic regulation at high densities \cite{dussutour2004optimal}. In lab experiments using narrow bridges connecting the nest and a food source, researchers found that ants can sustain a constant flow at high densities as opposed to the typical reduction in flow due to congestion in traffic models \cite{poissonnier2019experimental}. To understand such behavior, one may turn to microscopic or individual-based models that track every single ant in the colony and their interactions. In the case of ants, as customary in active matter systems, one typically accounts for the time evolution of the position and the orientation of each ant. Microscopic models of active matter have been used to study a vast array of different complex behaviors such as lane formation \cite{bacik2023lane,burger2016lane,sieben2017collective}, bird flocking \cite{carrillo2010asymptotic,chate2020dry,degond2013macroscopic,kruk2020traveling}, biological flows \cite{ohm2022weakly,saintillan2015theory} and clustering \cite{bruna2022phase,buttinoni2013dynamical,cates2015motility,liebchen2018synthetic}.

Ants utilize different pheromone molecules for various tasks such as trail formation, navigation, bridge building, nestmate recognition, and alarming in cases of danger; see, for example, \cite{blomquist2021chemical,czaczkes2015trail,david2009trail} and references therein.

Our starting point is an individual-based model proposed in \cite{dobramysl2023argentine}, following the experimental data of \cite{poissonnier2019experimental}. It consists of a set of stochastic differential equations (SDEs) for the position and orientation of each ant and a partial differential equation (PDE) describing the evolution of the pheromone concentration. The SDE model represents each ant as an active Brownian particle; that is, the position evolves according to a Brownian motion with a bias in the direction of the orientation, and the orientation changes gradually by a periodic Brownian motion. Ants change their orientations to align with the upward gradient of a pheromone field they lay. This mechanism is akin to that of autophoretic colloids studied in \cite{liebchen2017phoretic,pohl2014dynamic}, leading to a so-called Active Attractive Alignment (AAA) model in \cite{liebchen2019interactions}. These active colloids, which are synthetically manufactured and used as micro-engines and cargo carriers, are seen in experiments to undergo dynamic clustering even at slow densities of less than 10\% \cite{liebchen2019interactions}. The ant model in \cite{dobramysl2023argentine} differs from the AAA model in that the chemical sensing is not at the particle' center but at a \emph{look-ahead distance} that represents the location of the ants' antennas. We will show that this difference is critical in the type of instabilities the two models display.  

The SDE model described in \cite{dobramysl2023argentine} can be used to simulate a small number of ants and address questions such as whether ants can form and sustain trails, but the computations become expensive for large numbers. Interacting particle systems with many particles can alternatively be studied using macroscopic models. One way of obtaining a suitable macroscopic model is via a mean-field limit for a large number of particles. This framework has been used in many studies of collective behavior \cite{carrillo2014derivation}. Collective phenomena that arise via chemically interacting particles (chemotaxis) are an example successfully studied using mean-field models. A celebrated macroscopic model for bacterial chemotaxis is the Keller--Segel model \cite{Keller:1971eo}
\begin{subequations}
	\label{eq:kellersegel}
\begin{align}\label{eq:kellersegel1}
\frac{\partial \rho}{\partial t} & = \nabla\cdot[\nabla\rho - \chi \rho \nabla c],\\
\frac{\partial c}{\partial t} & = \Delta c + \rho -\alpha c, \label{eq:kellersegel2}
    \end{align}
\end{subequations}
where $\rho(t,\xx)$ is the bacteria population density, $c(t, \xx)$ is the chemical field, $\chi$ is the chemotactic sensitivity and $\alpha$ is the degradation  rate of the chemical attractant. Model \eqref{eq:kellersegel} was devised to study aggregation phenomena in bacterial populations, and the PDE was proposed on phenomenological grounds \cite{keller1970initiation, Keller:1971eo}. Subsequently, \eqref{eq:kellersegel} has been rigorously shown to be the limit as $N\to\infty$ of specific interacting particle systems. For example, Stevens \cite{stevens2000derivation} considered a stochastic system of moderately interacting particles as a starting point and obtained \eqref{eq:kellersegel} using a limiting procedure based on Oelschl\"ager \cite{Oelschlager:1985kv}. In the case that the left-hand side of \eqref{eq:kellersegel2} is set to zero $\partial_t c=0$, the Keller--Segel model has also been shown to be the mean-field limit of a system of weakly interacting particles at the level of PDEs for the averaged densities \cite{bresch2019mean} and by \cite{10.1214/16-AAP1267} at the empirical (random) density level. More recently, the results of \cite{10.1214/16-AAP1267} have been extended in  \cite{fournier2023particle} to the case  $\partial_t c\ne0$.

The behavior of the Keller--Segel model \eqref{eq:kellersegel} has been extensively studied. Initially, Keller and Segel derived a linear instability criterion for a more general model than (\ref{eq:kellersegel}) on a bounded domain with no-flux boundary conditions \cite{keller1970initiation}. In the notation of (\ref{eq:kellersegel}), the linear instability criterion (the \emph{Keller--Segel instability criterion}) is $\chi c_0>1$ for instability around the homogeneous steady states $\rho=\rho_0$ and $c=c_0$. Later, it was confirmed at a nonlinear level: when $\partial_t c=0$, if the mass $\int\rho \ \dd\xx$ is above a critical threshold depending on the domain geometry and $\chi$, solutions blow up in finite time. In contrast, below such mass, they will decay to zero (or to a homogeneous state on a bounded domain) \cite{blanchet2008infinite,blanchet2006two,horstmann2003,nagai2001blowup}. Note that this is equivalent to fixing $\int \rho \, {\rm d} \xx = 1$ and varying the chemical sensitivity $\chi$; this is the convention we will follow in this paper.

Versions of the Keller--Segel model with more terms were also considered for studying pattern formation in tissues and other microorganism aggregation phenomena. For example, one may include a Fisher-type competition term like $\rho(1-\rho)$ in the right-hand side of (\ref{eq:kellersegel}). This term generates multiple aggregation spots \cite{bellomo2015toward,painter2019mathematical}, instead of the collapse into one aggregation spot for the Keller--Segel model (\ref{eq:kellersegel}) with fewer terms. Another type of term that can be included is a diffusion parameter that depends on $\rho$. This is to model volume exclusion and mathematically prevent the model from blowing up \cite{bubba2020discrete,burger2006keller}. Furthermore, letting the chemical sensitivity depend on $\rho$ has been found to lead to traveling plateaus \cite{burger2008asymptotic}.

In this paper, we formally derive a PDE system for the ants' probability density $f(t, \xx, \theta)$ and the pheromone concentration $c(t,\xx)$ from the SDE model \cite{dobramysl2023argentine} using a mean-field approximation. From the equation for $f$, one can obtain a PDE for the spatial density $\rho(t, \xx) = \int f {\rm d} \theta$ and analyze it in the context of Keller--Segel type instabilities. In \cite{liebchen2017phoretic,pohl2014dynamic}, they show that the AAA model can produce immobile aggregates under suitable conditions and use a closure approximation to derive a PDE system for $\rho$ and $c$ of the form \eqref{eq:kellersegel} with effective diffusion and interaction parameters.  
From it, they obtain an analogous Keller--Segel instability criterion for an active chemotaxis system.
In \cite{fontelos2015pde}, they derive a PDE system from the AAA model phenomenologically and discuss the emergence of trails using a linear stability analysis.
Like the AAA model, the model we study is similar to the Keller--Segel model \eqref{eq:kellersegel}. In particular, pattern formation and blow-up phenomena, as studied for the Keller--Segel model, are useful frameworks for our model, and we can exploit some mathematical techniques developed for \eqref{eq:kellersegel}. However, since our model extends in phase space, it is substantially more complex and, as we will show, this results in more complex and higher dimensional patterning. Moreover, unlike other active matter models---which require a high enough chemotactic sensitivity \cite{blanchet2006two} or P\'eclet number \cite{cates2015motility} for instabilities to arise---our ant model displays clustering for arbitrarily small P\'eclet numbers. 

The structure of this paper is as follows. In Section \ref{sec:model}, we present the individual-based model of an ant colony and the associated macroscopic PDE model. We explain the look-ahead mechanism used to model ant antennas and show typical particle behavior for our model. Then, we present results for the macroscopic model in Section \ref{sec:results}. In Section \ref{sec:noblowup}, we prove that the solutions for this model do not blow up. Then, in Sections \ref{sec:uniqstat} and \ref{sec:nonlinstab}, for small enough interaction strength, we show that the stationary solution must be unique and nonlinearly stable in the linearly stable region, respectively. The linearly stable region is analyzed in Section \ref{sec:linstab} by means of numerical simulations. Lastly, the linearly unstable region is probed in Section \ref{sec:num} using a finite-volume method. We also compare the outcomes of the finite-volume method with the linear theory and provide a new quantification for lane formation.
%%%%%%%%%%%%%%%%%%%%%%%%%%%%%%%%%%%%%%%%%%%%%%%%%%%%%%%%%%%%
\section{Model}
\label{sec:model}
%%%%%%%%%%%%%%%%%%%%%%%%%%%%%%%%%%%%%%%%%%%%%%%%%%%%%%%%%%%%

\subsection{Microscopic model}\label{partmod}

We consider $N$ ants with positions and orientations $(\XX_i,\Theta_i)_{i=1}^N\in(\T^2_L\times\T_{2\pi})^N$, where $L$ is the length of the spatial periodic 2D box. The particles undergo the following multi-dimensional diffusion process \cite{dobramysl2023argentine}:
\begin{subequations}
	\label{sde_model}
\begin{align} 
    \dd \XX_i &= v_0 \bf{e}(\Theta_i)\dd t+\sqrt{2D_T}\dd \bf{W}_i^T,\label{eq:dX}\\
    \dd \Theta_i &= \gamma\bf{n}(\Theta_i)\cdot\nabla c_N(\XX_i+\lambda\bf{e}(\Theta_i))\dd t+\sqrt{2D_R}\dd W_i^R,\label{eq:dTheta}
\end{align}
where $\bf{e}(\theta)=(\cos\theta,\sin\theta)$ is the orientation vector, $\bf{n}(\theta) = (-\sin \theta, \cos \theta)$, and $\bf{W}_i^T$ and $W^R_i$ are independent periodic Brownian motions in two and one dimensions, respectively.
The constant $v_0$ represents the constant speed at which an ant moves along the direction ${\bf e}(\theta)$. This mechanism is the typical active transport term found in standard active matter models \cite{cates2015motility}. The constants $D_T$ and $D_R$ are the translational and rotational diffusion coefficients, respectively. The translational diffusion can be interpreted as a random force caused by small particle-environment or particle-particle volume interactions. The rotational diffusion can be construed as the baseline of rotational movement of the orientation of an individual ant in isolation. 
In \eqref{eq:dTheta}, $c_N(\xx)$ is a continuous chemical field that models the pheromone that the ants secrete and through which they interact with each other, $\gamma$ is the chemotactic sensitivity, and $\lambda$ is the look-ahead distance at which the particles detect the chemical gradient. This models ants having sensing antennas (see \cref{fig:lookahead}). 
\begin{figure}[htb]
    \centering
    \includegraphics[width=0.5\textwidth]{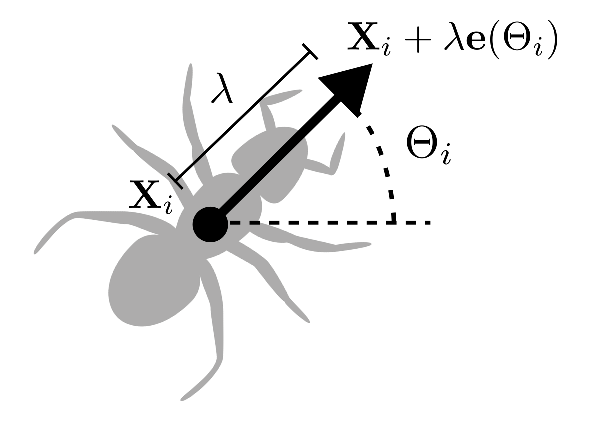}
    \caption{An explanation of the look-ahead distance mechanism.}
    \label{fig:lookahead}
\end{figure}
The chemical field causes a torque on the orientation $\Theta_i$, so ants align their orientation with the field's upward gradient $\nabla c_N$. 
The pheromone concentration $c_N$ satisfies the following equation
\begin{equation}
	    0= D\Delta c_N-\alpha c_N+\eta\rho_N, \label{eq:c0}
\end{equation}
\end{subequations}
where $D,\alpha$ and $\eta$ represent the diffusivity, decay and production rate, respectively, of $c_N$, and $\rho_N$ is the ants' spatial empirical measure, defined as
\begin{equation}\label{eq:empdens}
    \rho_N(t, \xx)=\frac{1}{N}\sum_{i=1}^N\delta(\xx - \XX_i(t)).
\end{equation} 

 The chemical field is modeled as instantaneously diffusing by setting $\partial_t c=0$. This reflects the pheromone diffusing much faster than the particle speeds \cite{liebchen2019interactions,liebchen2017phoretic,pohl2014dynamic}. For ants, this model is a relatively lightweight and short-lived pheromone that diffuses fast as used for navigation and trail formation \cite{czaczkes2015trail}.
 
\begin{remark}[Singularities of $c_N$]
A rigorous definition of \eqref{sde_model} poses challenges because the solution of the equation \eqref{eq:c0} for the concentration $c_N$ has singularities at positions $\xx=\XX_i(t)$.
To make this clear, note that one may write
\begin{equation}\label{eq:K0}
    c_N(t,\xx)=\frac{1}{N}\sum_{i=1}^N K_{\alpha,D}(\xx-\XX_i(t)),
\end{equation} 
where $K_{\alpha,D}(\xx)=K_0(\sqrt{\frac{\alpha}{D}}|\xx|)$ and $K_0:(0,+\infty)\to\R$ is the modified Bessel function of the second kind (see, e.g., \cite[\S II.3]{schwartz1966theorie}). The distance $|\cdot|$ here is the minimal Euclidean distance between two points on $\T_L^2$. 
Since $K_0$ blows up at the origin, we find that $c_N(t,\xx)$ is not defined at $\xx = \XX_i$ for $i = 1, \dots, N$. As a result, the alignment interaction or drift term in \cref{eq:dTheta} is non-Lipschitz with singularities whenever $\XX_i(t) + \lambda {\bf e}(\Theta_i(t)) = \XX_j(t)$, $i,j = 1, \dots, N$. If $\lambda > 0$, this can only happen if a second ant $j$ is exactly at the location of the first ant $i$ antennas, an event with probability zero. In contrast, if $\lambda = 0$, we must remove the self-interactions (cases $j = i$) as otherwise we would be evaluating $c_N$ in \cref{eq:dTheta} at a singular point. See \cref{mflimit} for more details on how to treat each case differently. 
This paper considers \cref{sde_model} as a formal SDE system. 
\end{remark}

\cref{fig:curvelambda} shows two sets of typical trajectories of the individual-based model \eqref{sde_model} for $N=8$ particles. We used a tamed Euler scheme for the SDEs \cite{hutzenthaler2012strong} and the expression \eqref{eq:K0} for $c_N$ (see Appendix \ref{sec:particlenumscheme} for details). Figure \ref{fig:curvelambda0} corresponds to setting the look-ahead parameter $\lambda$ in equation \eqref{eq:dTheta} to zero, and it thus coincides with the AAA model \cite{liebchen2019interactions}. It shows the typical Keller--Segel collapse: particles attract and collapse into clusters that remain stationary. Turning on the look-ahead mechanism ($\lambda>0$) results in a marked change to the system behavior (\cref{fig:curvelambda1}). Specifically, particles come together into a motile cluster with a persistent motion, characteristic of ant colony trails \cite{perna2012individual}.
\begin{figure}[htb]
\centering
    \subfloat[$\lambda=0$]{\includegraphics[width=0.4\textwidth]{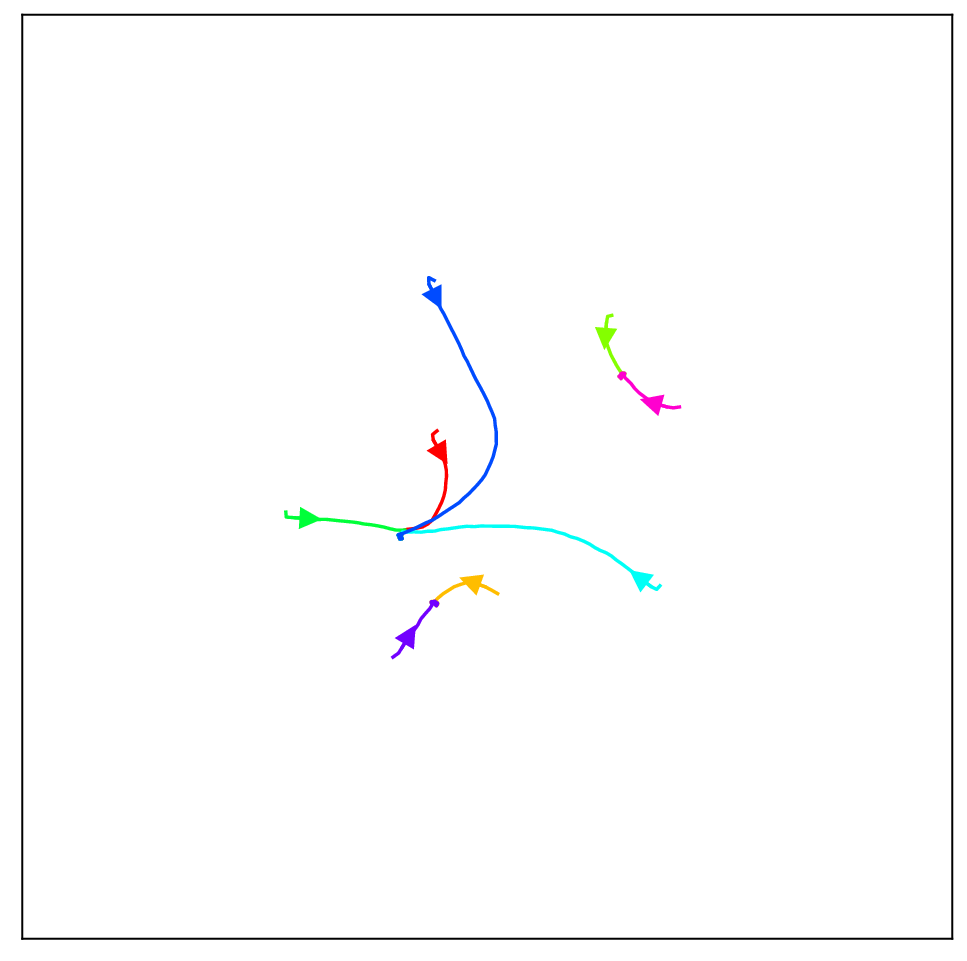}\label{fig:curvelambda0}}
    \subfloat[$\lambda=0.1$]{
    \includegraphics[width=0.4\textwidth]{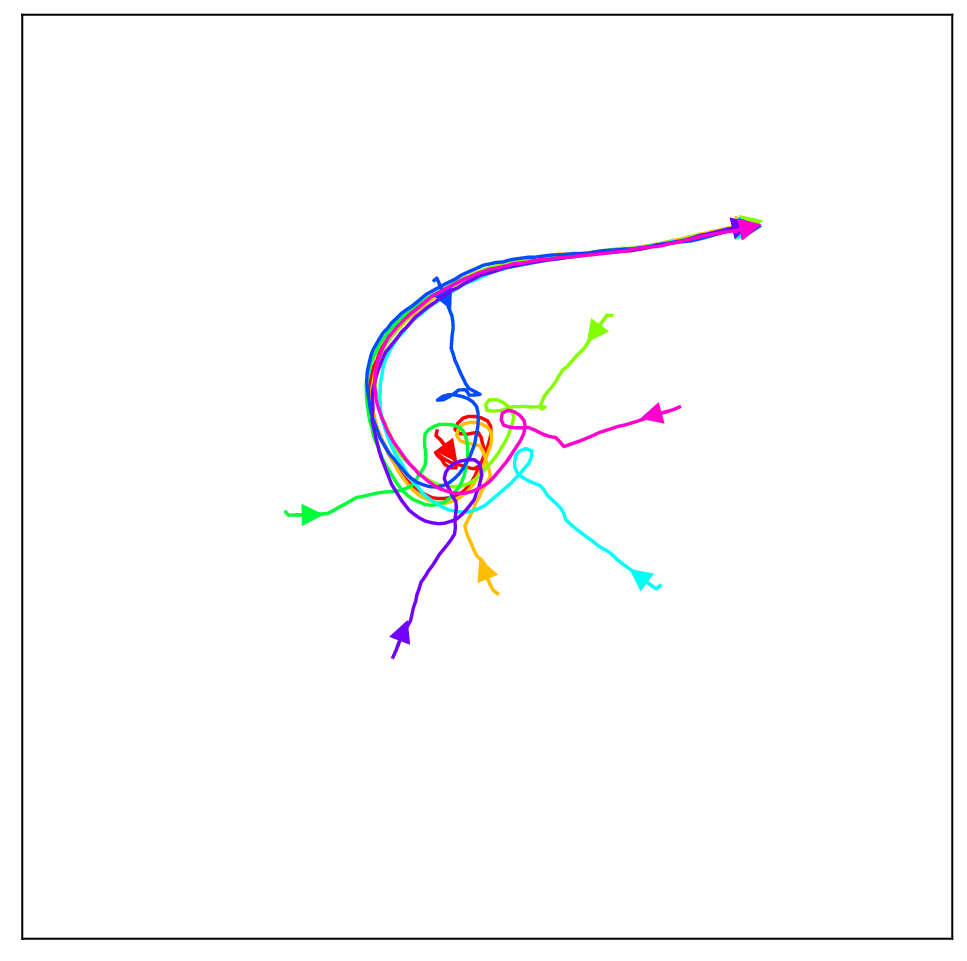}\label{fig:curvelambda1}}
\caption{Typical particle trajectories for particle simulations of the model with and without look-ahead mechanism. The arrows indicate the direction of the flow of time. Parameters: $N=8,L=D_R=D=\eta=\alpha=1.0,D_T=10^{-4},v_0=7.0,\gamma=300.0,\Delta t=10^{-5},t_{\max}=0.2$.}
\label{fig:curvelambda}
\end{figure}

We see that model \cref{sde_model} can reproduce the formation of trails without the need for an external prompt (like, e.g., target points in the domain modeling food sources or the nest \cite{amorim2015modeling}) or ``leader'' particles that the rest follow as in leader-follower models \cite{bernardi2021leadership}.
Trail formation in ants has also been reproduced using a microscopic lattice-based model \cite{edelstein1995trail} or through polar pheromone particles (so that the pheromone field contains information of the direction in which ants that laid it are moving) \cite{boissard2013trail,mokhtari2022spontaneous}.
External prompts are used in pedestrian lane formation models, such as a flux condition or an external drift field for the shortest path \cite{gomes2019parameter}. Leader-follower models are common in developmental biology, where leaders represent cells with a different phenotype than the majority  \cite{mclennan2015neural} and in studies of swarm dynamics such as for bee colonies, finding optimal parameters allowing leader particles to steer follower particles effectively to a target (\cite{bernardi2021leadership} and references therein). 

%%%%%%%%%%%%%%%%%%%%%%%%%%%%%%%%%%%
\subsection{Macroscopic model}
%%%%%%%%%%%%%%%%%%%%%%%%%%%%%%%%%%%%

We next present the macroscopic model, obtained by considering the mean-field limit of \eqref{sde_model}. One of our goals will be to study whether the resulting PDE system can reproduce the two distinct behaviors of the microscopic model observed in \cref{fig:curvelambda}. 

Let $f(t,\xx,\theta)$ be the macroscopic ant's probability density depending on the spatial position $\xx \in \T_L^2$, the orientation $\theta \in \T_{2\pi}$ and the time $t$. In the large particle limit $N\to\infty$, the formal mean-field limit of the SDE system \eqref{sde_model} is (see \cref{mflimit})
\begin{subequations} \label{eq:mfa}
\begin{align}
    \label{eq:f}
    \partial_t f&=\nabla_\xx\cdot[D_T\nabla_\xx f-v_0\bf{e}_\theta f]+\partial_\theta[D_R\partial_\theta f-\gamma\bf{n}_\theta \cdot\nabla_\xx c_\lambda f],\\
    0&=D\Delta_\xx c-\alpha c+\eta\rho,\label{eq:c}
\end{align}
\end{subequations}
where ${\bf e}_\theta = \bf{e}(\theta)$, ${\bf n}_\theta = \bf{n}(\theta)$,  $\rho(t,\xx)$ is the ant spatial density
 \begin{equation}\label{eq:rho}
    \rho(t,\xx)=\int_{0}^{2\pi}f(t,\xx,\theta)\dd\theta,
\end{equation}
$c(t, \xx)$ is the pheromone concentration and
\begin{equation}\label{eq:clambda}
	c_\lambda(t,\xx,\theta)=c(t,\xx+\lambda\bf{e}_\theta).
\end{equation} 
The ant density $\rho$ and the pheromone concentration $c$  correspond to the large $N$ limit of $\rho_N$ \eqref{eq:empdens} and $c_N$ \eqref{eq:K0} respectively. We are assuming $\int f\dd\xx\dd\theta=1$, since we are considering a probability density.

It is convenient to rescale time, space, and pheromone concentration to obtain a simplified version of \cref{eq:mfa}. In particular, we rescale time as  $t=T\hat{t}$, space as $\xx=L\hat{\xx}$ and chemical concentration as $c(\xx)=C_0\hat{c}(\hat\xx)$. We set $T=D_R^{-1},L=\sqrt{D/D_R}$ and $C_0=\eta/D_R$ and define the rescaled translational diffusion $\hat{D}_T=D_T/D$, the rescaled decay rate $\hat{\alpha}=\alpha/D_R$, the rescaled interaction strength $\hat{\gamma}=\eta\gamma/\sqrt{D D_R^3}$ and the Pecl\'et number $\Pe=v_0/\sqrt{D D_R}$. After dropping hats, the rescaled equations are 
\begin{subequations}\label{eq:fcresc}
    \begin{align}
    \partial_t f&=\nabla_\xx\cdot[D_T\nabla_\xx f-\Pe\bf{e}_\theta f]+\partial_\theta[\partial_\theta f-\gamma \bf{n}_\theta\cdot\nabla_\xx c_\lambda f],\label{eq:fresc}\\\
    0&=\Delta_\xx c-\alpha c+\rho.\label{eq:cresc}
\end{align}
\end{subequations}
Integrating equation (\ref{eq:fresc}) in $\theta$ gives \begin{equation}\label{eq:rho1}
    \partial_t\rho=\nabla\cdot[D_T\nabla\rho-\Pe\bf{p}],
\end{equation}
where \begin{equation}\label{eq:p}
    \bf{p}(t,\xx)=\int_0^{2\pi}\bf{e}_\theta f(t,\xx,\theta)\dd\theta,
\end{equation} is the polarisation. 
Comparing \cref{eq:rho1} with the Keller--Segel model \cref{eq:kellersegel}, we see that the polarisation $\bf{p}$ appears in the drift term.

%%%%%%%%%%%%%%%%%%%%%%%%%%
\section{PDE analysis}\label{sec:results}
%%%%%%%%%%%%%%%%%%%%%%%%%%

In this section, we study the macroscopic model \cref{eq:fcresc} analytically. We first define and show the existence of weak solutions up to time $T>0$. Then, we show that these solutions do not blow up at $t=+\infty$ in the $L^\infty$ norm. A corollary of this is that the solutions are unique.

The domain of positions and orientations is $\Sigma = \T^2\times\T_{2\pi}$, where $\T$ has length one and $\T_{2\pi}$ length $2\pi$. We use the notation $\xi\in\Sigma$ for the pair $(\bf{x},\theta)$. We will also use $\Omega$ for the spatial domain $\T^2$. For $1\leq p\leq+\infty$, we define  
\begin{align}
    L^p(\Sigma)=\{f:\Sigma\to\R \text{ measurable and periodic}|\|f\|_{L^p(\Sigma)}<+\infty\}.
\end{align} 
For $k=0,1,2,\dots$, $1\leq p\leq+\infty$ and $\beta\geq 0$ we define the Sobolev spaces $H^k(\Sigma)$, $W^{k,p}(\Sigma)$ and the H\"older spaces $C^{k,\beta}(\Sigma)$ analogously. The spaces in $\Omega$ are analogously defined. Unless explicitly written, the domain for integrals and function spaces is $\Sigma$.

We also define the Bochner space $L^p(0,T;\mathcal{X})$, for any Banach space $\mathcal{X}$, as 
\begin{equation}
    L^p(0,T;\mathcal{X})=\Bigg\{f:[0,T]\to\mathcal{X} \text{ measurable}\  | \int_0^T\|f(t)\|_{\mathcal{X}}^p\dd t<+\infty\Bigg\},
\end{equation} 
with the $L^p$ norm on $[0,T]$. We also use the notation $L^p_T\mathcal{X}$ for this space. Other regularity time-dependent spaces are analogously defined. For any Banach space $\mathcal{X}$, we define $\mathcal{X}'$ as its continuous dual space.
\begin{definition}[Weak solution of macroscopic model]
\label{def:weaksol}
    A weak solution to \cref{eq:fcresc} with initial data $f_\text{in}\in L^2$ up to time $T > 0$ is a function $f\in L^2(0,T;H^1)$, with $\partial_t f\in L^2(0,T;(H^1)')$,  such that for every $\varphi\in L^2(0,T;H^1)\cap L^2(0,T;L^\infty)$ 
\begin{multline}\label{eq:weaksol}
     \int_0^T\langle \partial_t f,\varphi\rangle\dd t=\int_0^T\Big[\langle D_T\nabla_\xx f,\nabla_\xx\varphi\rangle-\langle\Pe\bf{e}_\theta f,\nabla_\xx\varphi\rangle \\
     +\langle \partial_\theta f,\partial_\theta\varphi\rangle+\langle\gamma\partial_\theta(\bf{n}_\theta\cdot\nabla c_\lambda f),  \varphi\rangle \Big  ]\dd t,
\end{multline}
and $f(t=0)=f_\text{in}$ in $L^2$. The chemical field $c$ is defined as the unique strong solution in $L^2(0,T;H^2(\Omega))\cap H^1(0,T;H^1(\Omega))$ of $0=\Delta c-\alpha c+\rho$ for $\rho=\int_0^{2\pi} f\dd\theta \in L^2(0,T;L^2(\Omega)) \cap H^1(0,T;(H^1(\Omega))')$.
\end{definition}

Let us mention the somehow non-standard definition of the weak solution in the last term of \cref{eq:weaksol}, where we do not use integration by parts, since the derivatives of $c_\lambda$ and $f$ separate. More precisely, we have
\begin{equation}
    \partial_\theta(\bf{n}_\theta\cdot\nabla_\xx c_\lambda f) \varphi = (-\bf{e}_\theta\cdot\nabla_\xx c_\lambda   + \lambda \bf{n}_\theta D^2_\xx c_\lambda \bf{n}_\theta )  f \varphi  + \bf{n}_\theta\cdot\nabla_\xx c_\lambda \partial_\theta f \varphi. 
\end{equation}
Note that by standard results about Bochner spaces we have $ f \in L^\infty_T L^2$ and thus $c_\lambda \in L^\infty(0,T;H^2(\Omega))$. Together with $f, \partial_\theta f \in L^2_T L^2 ,\varphi \in  L^2_T L^\infty$ as well as the fact that $\bf{n}_\theta$ and $\partial_\theta \bf{n}_\theta $ are unit vectors, we see that the integrand is in $L^1$, i.e., the weak formulation is well-defined in the function spaces we use. 

%%%%%%%%%%%%%%%%%%%%%%%
\subsection{Well-posedness and no blow up}\label{sec:noblowup}
%%%%%%%%%%%%%%%%%%%%%%%

In this section, we will verify the well-posedness of the PDE model and show that there is no finite or infinite-time blow-up.
\begin{lemma}\label{lem:rhobound}
    Let $T>0$ and $f_\ini\in L^2$  be non-negative. Moreover, let $f$ be a non-negative weak solution of 
    \begin{equation} 
    \partial_t f =\nabla_\xx\cdot[D_T\nabla_\xx f-\Pe\bf{e}_\theta f]+\partial_\theta[\partial_\theta f+ vf],
    \end{equation}
    for some time-dependent scalar field $v$. Then the spatial density $\rho =  \int f d\theta$ satisfies
    \begin{equation}
        \int_\Omega \rho^2 \,\dd\xx \leq \e^{Ct} \int_\Omega \rho_\ini^2 \,\dd\xx,
    \end{equation}
     with $C= \frac{\Pe^2}{2D_T}$ and $\rho_\ini=\int f_\ini\dd\theta$. 
\end{lemma}
\begin{proof}
    Using a test function independent of $\theta$ we see that $\rho$ is a non-negative weak solution of  (\ref{eq:rho}), i.e.,
    \begin{equation}
        \partial_t \rho =\nabla_\xx\cdot\left[D_T\nabla_\xx \rho-\Pe \int_0^{2\pi} \bf{e}_\theta f~\dd\theta\right].
    \end{equation}
    Thus, a standard $L^2$ estimate yields 
    \begin{equation}\frac{\dd}{\dd t} \int_\Omega \rho^2 ~\dd\xx \leq - 2 D_T \int_\Omega|\nabla_\xx \rho|^2~\dd\xx +
    2\Pe \int_\Omega \nabla_\xx \rho \cdot\left(\int_0^{2\pi} \bf{e}_\theta f~\dd\theta \right)~\dd\xx.
    \end{equation}
    Now, it is easy to see that for non-negative $f$
    \begin{equation}
        \left\vert\int_0^{2\pi} \bf{e}_\theta f~\dd\theta \right\vert \leq \int_0^{2\pi}\vert \bf{e}_\theta \vert f~\dd\theta = \rho,
    \end{equation}
    and hence, Young's inequality implies 
    \begin{equation}
        \frac{\dd}{\dd t} \int_\Omega \rho^2 ~\dd\xx \leq \frac{\Pe^2}{2D_T} \int_\Omega \rho^2 ~\dd\xx.
    \end{equation}
    Finally, Gr\"onwall's lemma implies the assertion.   
\end{proof}

As a next step, we establish basic estimates for the density $f$ in three dimensions:
\begin{lemma} \label{lem:fbound}
Assume that, for initial value $f_\ini \in L^2$, there is a weak solution $f$ of
   \begin{equation}\label{eq:it}
         \partial_t f=\nabla_\xx\cdot[D_T\nabla_\xx f-\Pe\bf{e}_\theta f]+\partial_\theta[\partial_\theta f+ vf],
    \end{equation}
    with given $v \in L^2_T L^2$ such that 
    $$ \|\partial_\theta v\|_{L^2(0,T;L^2)} \leq M_0$$  
    Then, $f$ is bounded in $L^\infty(0,T;L^2) \cap L^2(0,T;H^1)$ with bounds depending only on $\Vert f_\ini \Vert_{L^2}$, $M_0$, $D_T$, and $\Pe$. Moreover, the weak solution $f$ is unique.
\end{lemma}
\begin{proof}
First of all, we establish for a.e. $t \in [0,T]$ 
\begin{align}
        \frac{1}{2}\frac{\dd}{\dd t}\int f^2 \dd\xi&=\int\left(-D_T|\nabla_\xx f|^2-|\partial_\theta f|^2+\tfrac{1}{2} \partial_\theta v f^2 \right) \dd\xi 
    \end{align} 
 which is formally obtained by multiplying (\ref{eq:it}) by $f$ and integrating by parts. To derive this identity rigorously, we use the bounded test function \begin{equation}
        \varphi^K = \max\{\min\{f,K\},-K\} ,
    \end{equation}
    and let $K$ tend to infinity. The nonlinear term can be estimated using Cauchy--Schwarz, the Sobolev embedding in dimension three \cite[Corollary 9.14]{brezis2010functional}, and Young's inequality, which gives 
\begin{align} \label{ineq:lambda4}
\begin{aligned}
	   \int\tfrac{1}{2}\partial_\theta v  f^2\dd\xi &\leq \tfrac{1}{2}\|\partial_\theta v\|_{L^2} \|f\|_{L^4}^2\\
  &\leq \tilde{C}\left[(\tfrac{1}{\varepsilon}+\varepsilon)\|\partial_\theta v\|_{L^2}^2\|f\|_{L^2}^2+\varepsilon\|\nabla_\xi f\|_{L^2}^2\right],
\end{aligned}
\end{align} 
for any $\varepsilon>0$ and some constant $\tilde{C}>0$. Therefore,  we find\begin{equation}
    \begin{aligned}
    \frac{1}{2}\frac{\dd}{\dd t}\int f^2 \dd\xi\leq&-D_T\|\nabla_\xx f\|_{L^2}^2-\|\partial_\theta f\|_{L^2}\\&+  \tilde C \left[(\tfrac{1}{\varepsilon}+\varepsilon)\|\partial_\theta v\|_{L^2}^2 \|f\|_{L^2}^2+\varepsilon\|\nabla_\xi f\|_{L^2}^2\right].
\end{aligned}
\end{equation}
Hence, by making $\varepsilon$ small enough, which only depends on the fixed constants $C_\alpha$, $D_T$, $\tilde{C}$, we get \begin{align}
    \frac{1}{2}\frac{\dd}{\dd t}\int f^2 \dd\xi\leq&-\tilde{C}_2\|\nabla_\xi f\|_{L^2}^2+\tilde{C}_3 \|\partial_\theta v\|_{L^2}^2 \|f\|_{L^2}^2 \leq \tilde{C}_3 \|\partial_\theta v\|_{L^2}^2  \|f\|_{L^2}^2,
\end{align} for constants $\tilde{C}_2>0$ and $\tilde{C}_3>0$. Thus, by Gr\"onwall's lemma we obtain a uniform bound on $\int f^2(t) \dd\xi$ in time (for $t \leq T$), namely
\begin{equation} \|f(t)\|_{L^2}^2 \leq \e^{\tilde{C}_3 M_0 } \|f_\ini\|_{L^2}^2.
\end{equation}
 Thus, we see that $f$ is bounded in $L^\infty(0,T;L^2)$ and, using the above estimate again without dropping the term multiplied by $\tilde C_2$, we also obtain boundedness of $f$ in $L^2(0,T;H^1)$.

 Note that the proof does not rely on the non-negativity of $f$. Hence, it can also be applied to the difference between two non-negative solutions and implies uniqueness of solutions. 
\end{proof}
In our setup, we choose the field $v$ to be
\begin{equation}
v=-\gamma\bf{n}_\theta\cdot\nabla_\xx c_\lambda, \quad  
\partial_\theta v = \gamma \bf{e}_\theta\cdot\nabla_\xx c_\lambda - \gamma \lambda \bf{n}_\theta\cdot(D^2_\xx c_\lambda) \bf{n}_\theta.
\end{equation}
Thus, we have
\begin{align}  \|\partial_\theta v\|_{L^2(0,T;L^2)} &\leq 
\gamma \|\bf{e}_\theta\cdot\nabla_\xx c_\lambda \|_{L^2(0,T;L^2)} + \gamma \lambda \| \bf{n}_\theta\cdot(D^2_\xx c_\lambda) \bf{n}_\theta \|_{L^2(0,T;L^2)} \nonumber\\
&\leq \gamma \|\nabla_\xx c_\lambda \|_{L^2(0,T;L^2)} + \gamma \lambda \| D^2_\xx c_\lambda \|_{L^2(0,T;L^2)}
\nonumber\\
&= \gamma \|\nabla_\xx c  \|_{L^2(0,T;L^2)} + \gamma \lambda \| D^2_\xx c \|_{L^2(0,T;L^2)}.
\end{align}

\begin{theorem}\label{lem:exist}
    For any $T>0$ and $f_\ini\in L^2$ non-negative, we have existence of a non-negative weak solution of equation (\ref{eq:f}) as in Definition \ref{def:weaksol}. Moreover, the weak solution satisfies $f\in C(0,T;L^2)$. 
\end{theorem}
\begin{proof}
    We define a fixed-point mapping ${\cal F}$ from $\check c$ to $c$ implicitly via solving the linear equation (\ref{eq:it})
  and $c$ being the solution of $0=\Delta c-\alpha c+\rho$  and $f(t=0)=f_\text{in}$. Here $\rho=\int_0^{2\pi} f\dd\theta$, where $f$ is a weak solution with $\check{c}$ such that $-\gamma\bf{n}_\theta\cdot\nabla_\xx \check{c}_\lambda=v$ for $v$ as in Lemma \ref{lem:fbound}. That is, we have a map \begin{equation}
      \mathcal{F}:L^2(0,T;H^2(\Omega))\to L^2(0,T;H^2(\Omega)), \ \check c\mapsto c.
  \end{equation}

More precisely, we investigate the fixed-point operator ${\cal F}$ in the strong topology of $L^2(0,T;H^2(\Omega))$, defined in the subset 
\begin{equation}
{\cal M} = \left \{c \in L^2(0,T;H^2(\Omega))~|~ c \geq 0, \Vert \nabla c \Vert_{L^\infty(0,T;L^2(\Omega))} +\Vert D^2 c \Vert_{L^\infty(0,T;L^2(\Omega))}\leq M \right \},
\end{equation}
with $M$ appropriately chosen (see below), depending on the bounds in Lemma \ref{lem:rhobound}. 

We first show existence and uniqueness of weak solutions $f$ given $\check c$, i.e., the well-definedness of the fixed-point operator. For this sake we first use a spatial smoothing $\check c^\epsilon$ of $\check c$, such that $\nabla \check c^\epsilon, D^2 \check c^\epsilon\in L^\infty(0,T;L^\infty(\Omega))$ and
\begin{equation}
\nabla \check c^\epsilon \rightarrow \nabla \check c,D^2\check c^\epsilon\to D^2\check c \quad \text{in}  \quad {L^\infty(0,T;L^2(\Omega))}\quad \text{as} \quad\epsilon\to 0,
\end{equation}
which implies existence and uniqueness of a non-negative weak solution $f^\epsilon$ by standard results on the linear Fokker--Planck equation \cite[Chapter 7, Theorem  3]{evans2010partial}. Now we can apply Lemma \ref{lem:fbound} to obtain uniform estimates for $f^\epsilon$, select a weakly converging subsequence and pass to the limit $\epsilon \rightarrow 0$ in the weak formulation. This implies the existence of a unique non-negative solution $f$ by Lemma 3.3. 

Similarly, we can now verify the continuity of the operator $\mathcal{F}$. Given a sequence $\check c_n$ in $\mathcal{M}$ converging in $L^2(0,T;H^2(\Omega))$ to some $\check c \in {\cal M}$ we establish the same uniform bounds for the solutions $f_n$ given $\check c_n$. By the Banach--Alaoglu theorem, we can extract weakly convergent subsequences of $f_n$ in $L^2(0,T;H^1) \cap L^\infty(0,T;L^2)$ and of $\partial_t f_n \in L²(0,T;(H^1)')$. For these, we can pass to the limit and obtain that the limit is a weak solution given $\check c$. Since the limit is unique, a standard argument implies weak convergence of the whole sequence $f_n$. For the corresponding spatial densities, Lemma \ref{lem:rhobound} allows us to derive a stronger result, which implies $\rho_n$ converges to $\rho$ in $L^2(0,T;H^1(\Omega)) \cap H^1(0,T;H^1(\Omega)') $. The Aubin--Lions--Simon theorem then implies strong convergence of $\rho_n$ to $\rho$ in $L^2$, and finally, the $L^2$ continuity of the solution operator of the Poisson equation implies the strong convergence of $c_n$ to $c$ in $L^2(0,T;H^2(\Omega))$.

To apply Schauder's fixed-point theorem and conclude the existence of a fixed point in ${\cal M}$, we need to show that ${\cal F}$ maps this closed set into a precompact subset. First of all, we use again Lemma \ref{lem:rhobound} and the continuity of the solution operator for the Poisson equation to conclude
\begin{equation}
   \Vert \nabla c \Vert_{L^\infty(0,T;L^2(\Omega))} +\Vert D^2 c \Vert_{L^\infty(0,T;L^2(\Omega))} \leq C_\alpha \Vert \rho \Vert_{L^\infty(0,T;L^2(\Omega))} \leq C_\alpha \e^{C T/2} \Vert \rho_\ini \Vert_{L^2(\Omega)},
\end{equation}
with the constant $C_\alpha$ depending only on $\alpha$. Choosing 
$M = C_\alpha \e^{C T/2} \Vert \rho_\ini\Vert_{L^2(\Omega)}$ we obtain the desired self-mapping. 

Finally, we use the identities
\begin{equation}
    -\Delta (\partial_t c) + \alpha (\partial_t c) = \partial_t \rho \in L^2(0,T;(H^1(\Omega))')
\end{equation}
and 
\begin{equation}
-\Delta (\partial_{x_i} c) + \alpha (\partial_{x_i} c) = \partial_{x_i} \rho \in L^2(0,T;L^2(\Omega))
\end{equation}
to conclude that 
\begin{equation}
c \in L^2(0,T;H^3(\Omega)) \cap H^1(0,T;(H^1(\Omega))').\end{equation}
We have the following continuous embeddings
\begin{equation}
    H^3(\Omega)\hookrightarrow H^2(\Omega)\hookrightarrow (H^1(\Omega))',
\end{equation}
 where $H^3(\Omega)\hookrightarrow H^2(\Omega)$ is also a compact embedding by the Rellich--Kondrachov theorem. Therefore, by the Aubin--Lions--Simon lemma
 \begin{equation}
    \{c\in L^2(0,T;H^3(\Omega)):\partial_t c\in L^2(0,T;(H^1(\Omega))')\}\subset L^2(0,T;H^2(\Omega))
\end{equation} as a compact embedding. This implies that ${\cal F}$ maps $\mathcal{M}$ into a precompact subset of ${\cal M}$ and hence the existence of a fixed point. The well-definedness of the map from $\check c = c$ to $f$ also implies the existence of a non-negative weak solution. 
\end{proof}
The well-posedness result of Lemma \ref{lem:exist} implies that the weak solutions exist globally in time. This result does not say what happens as $t\to+\infty$, e.g., whether there could be a blow-up at $t=+\infty$. We show that this is not possible in the following Proposition.
\begin{proposition}[No blow-up]
    Let $f$ be a weak solution to \cref{eq:fcresc} as in Definition \ref{def:weaksol} for some $T>0$, with non-negative initial data $f_{in}\in L^2 \cap L^\infty$. Then the solution extends up to any time $t\geq 0$ and $\sup_{t\geq 0}\|f(t)\|_{L^\infty}\leq C_\infty$ where $C_\infty$ depends on $D_T,\Pe,\gamma,\lambda,\alpha$ and $\|f_{in}\|_{L^\infty}$.
\end{proposition}
\begin{proof}
    The proof goes in multiple steps. The crucial element is an $L^p$ iteration method, iterating from $p$ to $p+1$ up to $p=+\infty$. This was first used by Alikakos for reaction-diffusion equations \cite{alikakos1979lp}. The fact that the solutions can be extended in time is already proven in Lemma \ref{lem:exist} and is a result of the $L^2$ estimate.
    
    The iteration method of \cite{alikakos1979lp} gives an upper bound for the quantity $\sup_{t\geq 0}\|u(t)\|_{L^\infty}$ where $u$ is some solution of a time-evolution PDE. We will use the version of Alikakos' method as in Lemma 5.1 of \cite{kowalczyk2005preventing}. We first apply Alikakos' method to $\rho$. Then, we show that this implies certain upper bounds for higher regularity norms of $c$. That is, we demonstrate that $\sup_{t\geq 0}\|\nabla c\|_{L^\infty(\Omega)}< C_c$, for some constant $C_c>0$. This then allows us to apply Alikakos' method to $f$ and thus obtain a bound for $\sup_{t\geq 0}\|f(t)\|_{L^\infty}$.
    
    Multiplying the equation for $\rho$ (\ref{eq:rho1}) by $(p+1)\rho^p$ gives \begin{equation}
        \frac{\dd}{\dd t}\int_\Omega\rho^{p+1}\dd\xx=-p(p+1)\int_\Omega \rho^{p-1}\nabla\rho\cdot[D_T\nabla\rho-\Pe\bf{p}]\dd \xx,
    \end{equation} using integration by parts. Using the identity $\frac{4}{p+1} \left|\nabla\rho^{\frac{p+1}{2}} \right|^2=(p+1)\rho^{p-1}|\nabla \rho|^2$, $|\bf{p}|\leq\rho$ almost everywhere and Young's product inequality we can derive \begin{align}
        \frac{\dd}{\dd t}\int_\Omega\rho^{p+1}\dd\xx\leq -\frac{4p}{p+1}(D_T-\Pe^2\epsilon)\int_\Omega \left|\nabla\rho^{\frac{p+1}{2}} \right|^2\dd\xx+\frac{1}{\epsilon}p(p+1)\int\rho^{p+1}\dd\xx,
    \end{align} for any $\epsilon>0$. We now use the Gagliardo--Nirenberg inequality in dimension two \cite[Chapter 9.C]{brezis2010functional} 
    \begin{equation}
        \|u\|_{L^2(\Omega)}^2\leq C_{GN}\|u\|_{L^1(\Omega)}\|u\|_{H^1(\Omega)},
    \end{equation} for some constant $C_{GN}>0$, to absorb the gradient terms into terms only involving $\rho$. This leads to \begin{equation}\label{ineq:lemma51}
        \frac{\dd}{\dd t}\int_\Omega\rho^{p+1}\dd x\leq -\epsilon_k\int_\Omega\rho^{p+1}\dd \xx+(a_k+\epsilon_k)2^{\beta k}\left[\int_\Omega\rho^{\frac{p+1}{2}}\dd\xx\right]^2,
    \end{equation} where $\epsilon_k=1/2^{qk},a_k=(2^k-1)2^k/\epsilon$ for $q$ sufficiently large and $\epsilon$ is chosen such that $D_T-\Pe^2\epsilon=\frac{1}{2}$. Therefore \cite[Lemma 5.1]{kowalczyk2005preventing} applies to $\rho$ and we get $\sup_{t\geq 0}\|\rho(t)\|_{L^\infty}\leq C_{\rho_\infty}\max\{1,\|\rho_\ini\|_{L^\infty}\}$ for some constant $C_{\rho_\infty}>0$.

    We need to demonstrate two things to apply Alikakos' method to $f$. Namely, a Gagliardo--Nirenberg inequality in dimension three and that the $L^\infty$ upper bound for $\rho$ implies an $L^\infty$ upper bound for $\nabla c$. Indeed, multiplying the equation for $f$ (\ref{eq:fresc}) by $(p+1)f^p$ gives \begin{align}
        \frac{\dd}{\dd t}\int f^{p+1}\dd\xi=&-p(p+1)\int\left[f^{p-1}D_T|\nabla_\xx f|^2+f^{p-1}|\partial_\theta f|^2-\gamma f^p\bf{n}_\theta\cdot\nabla c_\lambda \partial_\theta f\right]\dd\xi.
    \end{align} We see that if $\|\nabla c\|_{L^\infty}$ is uniformly bounded in time we can, using the Gagliardo--Nirenberg inequality to absorb gradient terms, cast this in the form of equation \cref{ineq:lemma51}. 
    
    That is, using similar identities as above, we first obtain
    \begin{equation}
    \begin{split}
        \frac{\dd}{\dd t}\int f^{p+1}\dd\xi=&-\frac{4p}{p+1}(\min\{D_T,1\}-\gamma^2\|\nabla c\|_{L^\infty}^2\epsilon)\int|\nabla_\xi f^{\frac{p+1}{2}}|^2\dd\xi+\\&\frac{1}{\epsilon}p(p+1)\int f^{p+1}\dd\xi,
    \end{split}
    \end{equation} and then application of the Gagliardo--Nirenberg inequality leads to \begin{equation}
        \frac{\dd}{\dd t}\int f^{p+1}\dd\xi\leq -\epsilon_k\int f^{p+1}\dd\xi+(a_k+\epsilon_k)2^{\beta k}\left[\int f^{\frac{p+1}{2}}\dd\xi\right]^2,
    \end{equation}
    with new $\epsilon_k,a_k,q$ and $\epsilon$ for $f$.

    In dimension three, the Sobolev embedding and H\"older $L^p$ interpolation inequality imply \begin{equation}
        \|u\|_{L^2}\leq C_S^{\frac{6}{5}}\|u\|_{L^1}^\frac{4}{5}\|u\|_{H^1}^{\frac{6}{5}},
    \end{equation} for some constant $C_S>0$ from the Sobolev embedding. This inequality will act as the Gagliardo--Nirenberg inequality did for applying the Alikakos' to $\rho$.

    For the $L^\infty$ upper bound of $\nabla c$, we employ Morrey's inequality with H\"older regularity exponent $1-\frac{2}{p}$ and a $W^{2,p}$ estimate for $c$ in terms of $\rho$. Hence, the exponent $p$ has to be larger than $2$ \cite[Theorem 9.11]{gilbarg1977elliptic}. More specifically, we use the chain of inequalities \begin{equation}
        \|\nabla c\|_{L^\infty(\Omega)}\leq C_M\|c\|_{W^{2,p}(\Omega)}\leq C_M C_{W^{2,p}}\|\rho\|_{L^\infty(\Omega)},
    \end{equation} for constants $C_M>0$ and $C_{W^{2,p}}>0$. Hence, Alikakos' method applies to $f$ and we can show that \begin{equation}
        \sup_{t\geq 0}\|f(t)\|_{L^\infty}\leq C_\infty,
    \end{equation} for some constant $C_\infty$ that depends on all the model's constants, $\|f_\ini\|_{L^\infty}$, $C_{GN}$, $C_S$, $C_M$ and $C_{W^{2,p}}$.
\end{proof}
Finally, we show that the $L^\infty$ bound leads to uniqueness of weak solutions.
\begin{corollary}
    Let $f$ be a weak solution to \cref{eq:fcresc} for non-negative $f_\ini\in L^2\cap L^\infty$. Then, it is unique.
\end{corollary}
\begin{proof}
    Uniqueness results from using the $L^\infty$ estimate for an $L^2$ estimate involving the difference between two weak solutions $f_1$ and $f_2$ with the same initial data. That is, 
    \begin{multline}
    	\frac{1}{2}\frac{\dd}{\dd t}\|f_1-f_2\|_{L^2}^2\\=\int \left [-D_T|\nabla_\xx (f_1-f_2)|^2-|\partial_\theta (f_1-f_2)|^2
        +\gamma \partial_\theta(b_1 f_1-b_2 f_2)(f_1-f_2) \right ]\dd \xi.
    \end{multline}
where  $b_i=\bf{n}_\theta\cdot\nabla(c_i)_\lambda$ for $i=1,2$. Using the identity $b_1 f_1-b_2 f_2=\tfrac{1}{2}(b_1+b_2)(f_1-f_2)+\tfrac{1}{2}(b_1-b_2)(f_1+f_2)$ and integration by parts gives 
\begin{multline}
	\frac{1}{2}\frac{\dd}{\dd t}\|f_1-f_2\|_{L^2}^2=\int \Big \{-D_T|\nabla_\xx (f_1-f_2)|^2-|\partial_\theta (f_1-f_2)|^2\\
        +\tfrac{1}{2}\gamma \left[(b_1+b_2)(f_1-f_2)+(b_1-b_2)(f_1+f_2) \right]\partial_\theta(f_1-f_2) \Big \}\dd \xi.
\end{multline}
    Now, we note that we can use the $L^\infty$ estimates for the terms $b_1+b_2$ and $f_1+f_2$ and the estimate \begin{equation}
        \|\nabla (c_1-c_2)\|_{L^2}\leq C_\alpha\|f_1-f_2\|_{L^2},
    \end{equation} so that 
    \begin{equation}
        \begin{aligned}
        \frac{1}{2}\frac{\dd}{\dd t}\|f_1-f_2\|_{L^2}^2\leq&-D_T\|\nabla_\xx (f_1-f_2)\|_{L^2}^2-\|\partial_\theta (f_1-f_2)\|^2_{L^2}\\
        &+\frac{C_{12}}{\varepsilon}\|f_1-f_2\|_{L^2}+\varepsilon C_{12}\|\partial_\theta(f_1-f_2)\|_{L^2}^2,
    \end{aligned}
    \end{equation}
     for some constant $C_{12}>0$ and any $\varepsilon>0$. Hence, by taking $\varepsilon$ small enough we get, by Gr\"onwall's lemma 
     \begin{equation}
        \|f_1(t)-f_2(t)\|_{L^2}^2\leq \e^{C_G t}\|f_1(t=0)-f_2(t=0)\|_{L^2}^2,
    \end{equation} for some $C_G>0$. Since $f_1(t=0)=f_2(t=0)=f_\ini$, we see that $f_1(t)=f_2(t)$ in $L^2$ for all $t\geq 0$.
\end{proof}
\subsection{Conditional uniqueness of the stationary state}\label{sec:uniqstat}
Stationary solutions for equations (\ref{eq:fcresc}) tell us about the possible endstates the solutions may converge to. Considering $D_T,\Pe,\lambda$ and $\alpha$ as fixed parameters, we show that for small enough $\gamma$ the homogeneous state $f_\ast=\frac{1}{2\pi}$ is the only stationary solution. Therefore, non-trivial stationary patterns do not exist in this parameter region.
\begin{proposition}[Conditional uniqueness of the stationary solution]\label{prop:uniqstat}
    Assume $f$ is a non-negative weak stationary solution as in \cref{def:weaksol}. Then, for fixed $D_T,\Pe,\lambda$ and $\alpha$, if $\gamma$ is small enough, it follows that $f$ is unique and hence equal to the homogeneous solution $f_\ast=\frac{1}{2\pi}$.
\end{proposition}
\begin{proof}
    The main idea of the proof is to show a contradiction if $f_1$ and $f_2$ are two distinct solutions to the stationary problem. To arrive there, we employ Sobolev embeddings, H\"older inequalities, the $W^{2,p}$ inequality used earlier and $L^1$ elliptic regularity theory. The main inequality that will lead to a contradiction, if $\gamma$ is small enough, is the $L^2$ estimate for the difference $f_1-f_2$. That is, assuming $f_1$ and $f_2$ are two distinct non-negative solutions of  the stationary equation \begin{equation}\label{eq:stat}
        0=D_T\Delta_\xx f_i-\Pe\bf{e}_\theta\cdot\nabla_\xx f_i+\partial_\theta^2 f_i-\gamma\partial_\theta(\bf{n}_\theta\cdot\nabla_\xx(c_i)_\lambda f_i),
    \end{equation} we subtract their two equations and multiply by $f_1-f_2$ to get, using integration by parts, 
    \begin{equation}
        \begin{aligned}
        &D_T\|\nabla_\xx(f_1-f_2)\|_{L^2}^2+\|\partial_\theta(f_1-f_2)\|_{L^2}^2=-\gamma\int\partial_\theta(b_1f_1-b_2 f_2)(f_1-f_2)\dd\xi\\
        &=\gamma\int\tfrac{1}{2}\left[(b_1-b_2)(f_1+f_2)+(f_1-f_2)(b_1+b_2)\right]\partial_\theta(f_1-f_2)\dd\xi,\label{ineq:f1f2}
    \end{aligned}
    \end{equation}
where again we write $b_i=\bf{n}_\theta\cdot\nabla (c_i)_\lambda$. To show that this equality gives a contradiction for small enough $\gamma$, we establish upper bounds for $\|f_i\|_{L^\infty}$ and $\|b_i\|_{L^\infty}$. These upper bounds allow us to turn (\ref{ineq:f1f2}) into, by using Young's product inequality and Poincar\'e's inequality for $f_1-f_2$, 
     \begin{equation}
        \min\{D_T,1\}\|\nabla_\xi(f_1-f_2)\|_{L^2}^2\leq\frac{1}{2}\gamma F_\infty C'\|\nabla_\xi(f_1-f_2)\|_{L^2}^2,\label{ineq:f1f22}
    \end{equation}
for some constants $F_\infty, C'>0$. The constants $F_\infty$ and $C'$ depend on $\gamma$ (and other parameters we consider fixed), and they can be made arbitrarily small by making $\gamma$ smaller (they are non-decreasing in $\gamma$). Hence, if $\|\nabla_\xi(f_1-f_2)\|_{L^2}>0$ ($f_1$ and $f_2$ are distinct) and $\gamma$ is small enough then (\ref{ineq:f1f22}) gives a contradiction and so $f_1$ must be equal to $f_2$.

We finish the proof by showing that $\|f_i\|_{L^\infty}$ and $\|b_i\|_{L^\infty}$ can be bounded from above. To do this, we first show that $\|\rho\|_{L^\infty}$ (dropping the $i$ index for the moment) can be bounded from above. Using the identity $D_T\Delta\rho=\Pe\nabla\cdot\bf{p}$ and the Sobolev embedding for H\"older regularity, we arrive at
\begin{equation}
        \|\rho\|_{L^\infty(\Omega)}\leq C_{H,2}\|\rho\|_{H^2(\Omega)}\leq C_{H,2} \left(2\frac{\Pe}{D_T}+1\right)\|f\|_{H^1},
\end{equation} 
for some constant $C_{H,2}>0$.

Multiplying equation (\ref{eq:stat}) by $f$ and standard estimates for the Poisson equation for $c$ lead to \begin{equation}
    \min\{D_T,1\}\int|\nabla_\xi f|^2\dd\xi\leq \tfrac{1}{2}\gamma\int |\partial_\theta b|f^2\dd\xi\leq 2\gamma\|\rho\|_{L^2}\|f\|_{L^4}^2.
\end{equation} 
Then, by applying a Sobolev embedding, H\"older inequalities and Young's inequalities we can conclude \begin{equation}\label{ineq:fh1}
        \min\{D_T,1\}\|\nabla_\xi f\|_{L^2}^2\leq C_\gamma (\|f\|_{L^2}^6+\|f\|_{L^2}^2),
\end{equation} 
for a constant $C_\gamma$ that depends non-decreasingly on $\gamma$.

Next, we use the identity $D_T\Delta_\xx f+\partial_\theta^2 f=\Pe\bf{e}_\theta\cdot\nabla_\xx f+\gamma\partial_\theta(b f)$ and evaluate both sides in $L^2$ norm. The right-hand side can be upper bounded in terms of $\|f\|_{L^2}$ using (\ref{ineq:fh1}) and the fact that $\|D^2 c\|_{L^4}$ can be bounded from above by terms depending on $\|f\|_{L^2}$. The latter follows by applying the $W^{2,p}$ estimate from \cite[Theorem 9.11]{gilbarg1977elliptic} with $p=4$ and that $\|f\|_{L^4}$ is controlled by $\|f\|_{L^2}$ via a Sobolev embedding and by inequality (\ref{ineq:fh1}). Then, using the Sobolev embedding for H\"older regularity for $f$ gives 
\begin{equation}
        \|f\|_{L^\infty}\leq C_{H,3}\|f\|_{H^2}\leq G(\gamma,\|f\|_{L^2}),
\end{equation} 
for some constant $C_{H,3}>0$ and some non-decreasing function $G$ that captures all the nonlinearities in the inequalities above.

    Lastly, we show that $\|f\|_{L^2}$ can be bounded from above by $\|f\|_{L^1}=1$ using a $W^{1,p}$ Sobolev embedding and $L^1$ elliptic regularity theory. We have \cite[Corollary 2.8]{veron2008elliptic}
    \begin{equation}
        \|f\|_{W^{1,\frac{3}{2}-\epsilon}}\leq C_{L^1}\|f\|_{L^1},
    \end{equation} for some small $\epsilon>0$ and a constant $C_{L^1}>0$. Using H\"older interpolation for $L^2$ in terms of $L^1$ and $L^3$ then gives that $\|f\|_{L^2}$ can be bounded from above depending on $C_{L^1}$ and $\|f\|_{L^1}$. We note that, by definition, we have $\|f\|_{L^1}=1$. This finishes the proof.
\end{proof}
\subsection{Conditional nonlinear stability of the homogeneous state}\label{sec:nonlinstab}
We consider the stability of the homogeneous state $f_\ast$ by studying the equation for a perturbation $\tilde{f}$, such that $f=f_\ast+\tilde{f}$ is the full solution. Analogously, we write $\rho = \rho_\ast + \tilde \rho$ and $c = c_\ast + \tilde c$. Note that $f_\ast = 1/(2\pi)$, $\rho_\ast = 1$ and $c_\ast = 1/\alpha$. The equation for the perturbation $\tilde{f}$ is
\begin{equation} \label{eq_pert}
    \partial_t \tilde{f}=L\tilde{f}-\gamma\partial_\theta(\bf{n}_\theta\cdot\nabla_\xx \tilde{c}_\lambda \tilde{f}),
\end{equation} where $L$ is the linear operator \begin{equation}\label{eq:linop}
    L\tilde{f}=D_T\Delta_\xx \tilde{f}-\Pe\bf{e}_\theta\cdot\nabla_\xx \tilde{f}+\partial_\theta \tilde{f}+f_\ast\gamma\partial_\theta(\bf{n}_\theta\cdot \nabla_\xx \tilde{c}_\lambda),
\end{equation} 
and the equation for $\tilde c$ is $0=\Delta \tilde{c}-\alpha \tilde{c}+\tilde{\rho}$.

In this section, we show that the equations are nonlinearly stable around the homogeneous state $f_\ast$ in a subset of the stable linear regime of the linear operator $L$. That is, we assume the linear operator $L$ is stable in the sense that it satisfies an operator decay estimate \begin{equation}\label{eq:decest}
    \|\e^{tL}\|_{L^2}\leq \e^{-\beta t},
\end{equation} for some $\beta>0$. Here $\e^{tL}$ is the semigroup generated by the operator $L$. For our operator, this is equivalent to the eigenvalue $\sigma$ of $L$ with the largest real part being strictly negative, $\Re(\sigma)<0$. In this case, $\beta=-\Re(\sigma)$ \cite{kato2013perturbation,renardy2004introduction}. The space in which we show decay is the space $H^2_\xx L^2_\theta$ which is defined as \begin{equation}
    H^2_\xx L^2_\theta=\{f\in L^2:\|f\|_{L^2}+\|\nabla_\xx f\|_{L^2}+\|D^2_\xx f\|_{L^2}<+\infty\}.
\end{equation}
\begin{proposition}[Nonlinear stability of small perturbations around the homogeneous state in the linear stable region and for small $\gamma$]\label{prop:nonlinstab}
    Let $\tilde{f}$ be the perturbation such that the full solution is $f=f_\ast+\tilde{f}$, where $f$ is a weak solution as in Definition \ref{def:weaksol}. Suppose $f(t)\in H^2_\xx L^2_\theta$ for all $t\geq 0$ and assume that the operator $L$ (\ref{eq:linop}) satisfies a decay estimate (\ref{eq:decest}) for some $\beta>0$. Then, for any $\delta_0<\beta$, if $\|\tilde{f}_\ini\|_{H^2_\xx L^2_\theta}$ is small enough and $\gamma<2f_\ast\min\{D_T,1\}/\alpha$, the perturbation decays $\|\tilde{f}(t)\|_{H^2_\xx L^2_\theta}\leq \tilde{C} \e^{-\delta_0 t}$ and the constant $\tilde{C}>0$ depends on the model constants, $\delta_0$ and $\|\tilde{f}_\ini\|_{H^2_\xx L^2_\theta}$.
\end{proposition}
\begin{proof}
    The idea for the proof is similar to \cite[Theorem 1.3]{albritton2022stabilizing}. We drop the tilde for $\tilde{f}$. Using the Duhamel formulation \begin{equation}
        \tilde{f}(t)=\e^{tL}\tilde{f}(t=0)-\gamma\int_0^t\e^{(t-s)L}\partial_\theta((\bf{n}_\theta\cdot\nabla_\xx \tilde{c}_\lambda)\tilde{f})\dd s,
    \end{equation} we show a bootstrap argument with the norm \begin{equation}
        \|f\|_{X_T}=\sup_{t\in[0,T]}\e^{\delta_0 t}\|f(t)\|_{H^2_\xx L^2_\theta},
    \end{equation} for some $\delta_0$ to be determined.
    
    To do so, we show a duality argument for the nonlinear term to turn an estimate on the operator $\e^{(T-s)L}\partial_\theta$ into one on $-\partial_\theta \e^{(T-s)L^\ast}$. That is, we consider the operator \begin{equation}
        \textsf{T}:L^2_w((0,T)\times\Sigma)\to L^2(\Sigma), \quad g\mapsto\int_0^T\e^{(T-s)L}\partial_\theta g(s)\dd s,
    \end{equation} where we use the space $L^2_w((0,T)\times\Sigma)$ with norm \begin{equation}
        \|g\|_{L^2_w}^2=\int_0^T w(t)^2\|g(t)\|_{L^2}^2\dd t,
    \end{equation} for some integrable function $w:(0,T)\to\R$, called the weight function. We set the weight function to be $\e^{\beta s}$. The dual of $\textsf{T}$ is \begin{equation}
        \textsf{T}^\ast:L^2(\Sigma)\to L^2_{w^{-1}}((0,T)\times\Sigma),\quad h\mapsto-\partial_\theta \e^{(T-s)L^\ast}f,
    \end{equation} where \begin{equation}\label{eq:Last}
        L^\ast f=D_T\nabla_\xx f+\Pe\bf{e}_\theta\cdot\nabla_\xx f+\partial_\theta^2 f+f_\ast\gamma\nabla_\xx\cdot \text{I} \textsf{S}(\bf{n}_\theta \textsf{H}_{-\lambda}\partial_\theta f),
    \end{equation} and we defined the operators \begin{align}
        &\text{I}:L^2(\Sigma)\to L^2(\Sigma),f\mapsto((\xx,\theta)\mapsto\int_0^{2\pi} f(\xx,\theta')\dd\theta'), \\ &\textsf{S}:L^2(\Sigma)^2\to L^2(\Sigma)^2,\bf{f}\mapsto(\alpha\Id-\Delta_\xx)^{-1}\bf{f},\\
        &\textsf{H}_{\lambda}:L^2(\Sigma)\to L^2(\Sigma),f\mapsto f(\xx+\lambda\bf{e}_\theta,\theta).
    \end{align}
    The term $-\e^{(T-s)L^\ast} h$ is a solution to the PDE problem \begin{equation}
        \partial_s f=L^\ast f, \quad f(T)=h.
    \end{equation}
    Hence, multiplying \cref{eq:Last} by $f$ gives \begin{equation}\label{eq:energy1}
        \frac{1}{2}\frac{\dd}{\dd s}\|f(s)\|_{L^2}^2=-D_T\|\nabla_\xx f\|_{L^2}^2-\|\partial_\theta f\|_{L^2}^2-f_\ast\gamma\langle \text{I} \textsf{S}(\bf{n}_\theta \textsf{H}_{-\lambda}\partial_\theta f),\nabla_\xx f\rangle.
    \end{equation}
    By making use of $\|\text{I}\|=\|\textsf{H}_\lambda\|=1$ and $\|\textsf{S} \bf{f}\|_{L^2}\leq\alpha^{1/2}\|\bf{f}\|_{L^2}$ we get \begin{align}
        &\frac{1}{2}\frac{\dd}{\dd s}(\e^{-2\beta s}\|f(s)\|_{L^2}^2)\leq -\e^{-2\beta s}(\min\{D_T,1\}-\tfrac{1}{2}f_\ast\gamma\alpha)\|\nabla_\xi f(s)\|_{L^2}^2,\\
        &(2\min\{D_T,1\}-f_\ast\gamma\alpha)\int_0^T\e^{-2\beta s}\|\nabla_\xi f(s)\|_{L^2}^2\leq \e^{-2\beta T}\|h\|_{L^2}^2,
    \end{align} where we also employed the dual decay estimate $\|\e^{(T-s)L^\ast}h\|_{L^2}\leq \e^{-\beta(T-s)}\|h\|_{L^2}$. Therefore, by using duality,  \begin{equation}\label{ineq:Tast}\|\textsf{T}\|=\|\textsf{T}^\ast\|\leq\frac{\e^{-\beta T}}{\sqrt{(2\min\{D_T,1\}-f_\ast\gamma\alpha)}},\end{equation}
and we see that we require $\alpha f_\ast\gamma<2\min\{D_T,1\}$.

Now, we set up the bootstrap argument by assuming \begin{equation}
    \|f\|_{X_T}\leq C_0\epsilon,
\end{equation} and we want to show that this implies $\|f\|_{X_T}\leq\frac{1}{2}C_0\epsilon$ for certain $\delta_0,C_0$ and $\epsilon>0$, so that then $T$ can be extended to any $T>0$ and $\sup_{t\geq 0}\e^{\delta_0 t}\|f(t)\|_{H^2_\xx L^2_\theta}<C_0\varepsilon$, concluding the proof. The duality argument gives \begin{equation}
    \left\|\int_0^T\e^{(T-s)L}\partial_\theta (\bf{n}_\theta\cdot\nabla_\xx c_\lambda f)\dd s\right\|_{H^2_\xx L^2_\theta}^2\leq \frac{\e^{-2\delta_0 T}}{(2\min\{D_T,1\}-f_\ast\gamma\alpha)2(\beta-\delta_0)}C_0^4\epsilon^4,
\end{equation} using that the inequality \cref{ineq:Tast} generalises to $H^2_\xx L^2_\theta$ by using Fourier multipliers in $\xx$ and that $H^2_\xx$ is an algebra. Therefore, choosing any $\delta_0<\beta$ we get \begin{equation}
    \epsilon+\frac{f_\ast\gamma C_0^2\epsilon^2}{\sqrt{(2\min\{D_T,1\}-f_\ast\gamma\alpha)2(\beta-\delta_0)}}\leq\frac{1}{2}C_0\epsilon,
\end{equation} for the bootstrap argument to work. For example, for $C_0\geq 2$, this inequality is satisfied for small enough $\epsilon$. This concludes the proof.
\end{proof}
We note that energy estimate (\ref{eq:energy1}) is unfavourable and obstructs the proof extending to the entire linearly stable region. The operator $L$ is not self-adjoint and does not enjoy a more favourable energy estimate as in \cite{albritton2022stabilizing}. We note that the requirement that $\gamma$ is small enough for fixed other parameters is somewhat similar to the requirement for Proposition \ref{prop:uniqstat}, but depending on other constants coming from regularity estimates.

%%%%%%%%%%%%%%%%%%
\section{Numerical results}
%%%%%%%%%%%%%%%%%%

In this section, we study the emergence of non-trivial solutions of the PDE model \cref{eq:fcresc} by means of a linear stability analysis and time-dependent simulations. The numerical codes used throughout this section are available freely at \url{https://github.com/odewit8/ants_repo}.

%%%%%%%%%%%%%%%%%%
\subsection{Linear stability of the homogeneous state}\label{sec:linstab}
%%%%%%%%%%%%%%%%%%

We neglect the higher-order terms in \cref{eq_pert} and consider the linear problem for a perturbation $\tilde f$ around the homogeneous state $f_\ast$, $\partial_t \tilde f = L \tilde f$, where recall that the linear operator $L$ \cref{eq:linop} depends on $\tilde c_\lambda = \tilde c(t, \xx + \lambda \bf{e}_\theta)$. 
By linearising $L$ up to order one in $\lambda$, $L = L_\lambda + O(\lambda^2)$, we can resolve the spectrum of $L_\lambda$. This allows us to gain insight into the spectrum of $L$ itself and indicate where we expect a decay estimate \cref{eq:decest} to hold. Linearising \cref{eq:linop} in $\lambda$ gives the operator 
\begin{equation}
    L_\lambda \tilde f=D_T\Delta_\xx \tilde f-\Pe\bf{e}_\theta\cdot\nabla_\xx \tilde f+\partial_\theta^2 \tilde f+f_\ast\gamma\bf{e}_\theta\cdot\nabla_\xx \tilde c-\lambda f_\ast\gamma\begin{bmatrix}
        \partial_y^2 \tilde c-\partial_x^2 \tilde c\\-2\partial_{xy} \tilde c
    \end{bmatrix}\cdot\bf{e}_{2\theta}.
\end{equation}
We note that the spectrum of $L_\lambda$ (and $L$) consists entirely of eigenvalues because they have compact resolvent \cite{kato2013perturbation}.

The linear stability analysis considers solutions $\tilde f \sim e^{\sigma t}$, leading to the eigenvalue problem $L_\lambda \tilde f=\sigma \tilde f$.  To find the eigenvalue $\sigma$ with the largest real part, it suffices to consider a wave along one axis at the lowest non-trivial frequency. That is, it suffices to consider eigenfunctions of the form
\begin{equation}\label{ansatz}
    \tilde{f}=\e^{\sigma t+\ii\omega x}\sum_{k\geq 0}A_k\cos(k\theta)
\end{equation}
where $\omega = 2\pi$ and $A_k\in \mathbb C$. 

We are interested in the parameter region in the $\Pe$-$\gamma$ plane for which $\Re(\sigma)=0$, demarcating the transition from the linearly stable region to the linearly unstable region, also called the linear instability line. From \eqref{ansatz} we have that $\rho = 2\pi A_0\e^{\sigma t+\ii\omega x}$ and, using \eqref{eq:cresc},  the chemical field can be written as $c(t,x)=\frac{2\pi A_0}{\omega^2+\alpha}\e^{\sigma t +\ii\omega x}$. Then, the eigenvalue problem is reduced to
\begin{subequations}
	\label{eigprob}
\begin{align}
    \sigma A_0&=-\omega^2 D_T A_0-\tfrac{1}{2}\ii\omega \Pe A_1,\\
    \sigma A_1&=-\omega^2 D_T A_1- A_1-\ii\omega \Pe A_0-\tfrac{1}{2}\ii\omega \Pe A_2+\ii \gamma\omega \frac{A_0}{\omega^2+\alpha},\\
    \sigma A_2&=-\omega^2 D_T A_2-4 A_2-\tfrac{1}{2}\ii\omega \Pe(A_1+A_3)-\lambda \gamma\omega^2\frac{A_0}{\omega^2+\alpha},\\
     \sigma A_k &=-\omega^2 D_T A_k-k^2 A_k-\tfrac{1}{2}\ii\omega \Pe(A_{k-1}+A_{k+1}),
\end{align}
\end{subequations}
for $k\geq 3$,
where we have used $f_* = 1/2\pi$ and
\begin{equation*}
    2\cos(\theta)\cos(k\theta)=\cos((k+1)\theta)+\cos((k-1)\theta),
\end{equation*} for any integer $k$. 

We now follow a procedure similar to \cite{10.1098/rspa.2023.0524} to obtain a dispersion relation $\Re(\sigma)=0$ in terms of $\gamma$ and $\Pe$. 
Specifically, we rewrite the recurrence \eqref{eigprob}  in terms of a countable banded matrix $M$, that is, 
\begin{equation}\label{eq:evproblem}
    \begin{pmatrix}
        a & b & 0 & 0 & \hdots\\
        iC_1+2b & a-1 & b & 0 & \hdots\\
        -\lambda\omega C_1 & b & a-4 & b  & \hdots\\
        0 & 0 & b & a-9  & \hdots\\
        \vdots & \vdots & \vdots &\vdots  & \ddots
    \end{pmatrix}\begin{pmatrix}
        A_0\\ A_1 \\ A_2 \\ A_3 \\ \vdots
    \end{pmatrix}=\sigma \begin{pmatrix}
        A_0\\ A_1 \\ A_2 \\ A_3  \\ \vdots
    \end{pmatrix},
\end{equation}
where \begin{equation} \label{params_disp}
    a=-\omega^2 D_T,\qquad b=-\mathrm{i}\tfrac{1}{2}\omega \Pe,\qquad C_1=\frac{\gamma\omega}{\omega^2+\alpha}.
\end{equation} 
The idea is to consider the $n$-dimensional system resulting from truncating $M$ to the $n\times n$ matrix $M_n$ and exploit the fact that, as $n\to \infty$, the eigenvalue $\sigma_n^\text{max}$ with the largest real part of $M_n$ converges extremely fast to $\sigma^\text{max}$ \cite{ikebe1996eigenvalue}. \cref{fig:dispersion}  plots $\text{Re}(\sigma_n^\text{max}) = 0$ for $\lambda = 0, 0.1$ and various values of $n$, displaying the rapid convergence to $\text{Re}(\sigma^\text{max}) = 0$, the boundary between stability and instability of the linearised problem $\partial_t \tilde f = L_\lambda \tilde f$. The scheme has already converged for $n=2$ in the case of $\lambda = 0$, while it requires $n= 8$ to converge when the look-ahead parameter is $\lambda = 0.1$.
To the right of the converged lines lies the linearly unstable region; that is, the homogeneous state becomes unstable for large enough $\gamma$. 
\begin{figure}[htb]
    \centering
    \subfloat[$\lambda=0$]{\includegraphics[width=0.49\textwidth]{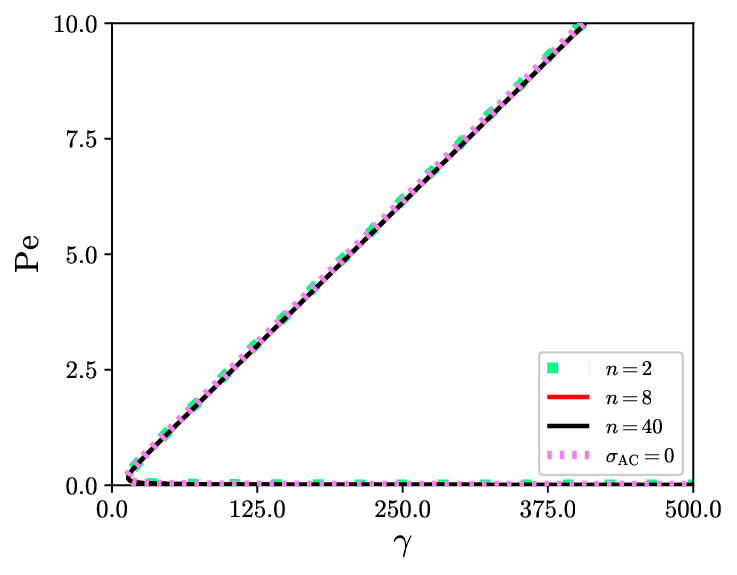}}
    \subfloat[$\lambda=0.1$]{\includegraphics[width=0.49\textwidth]{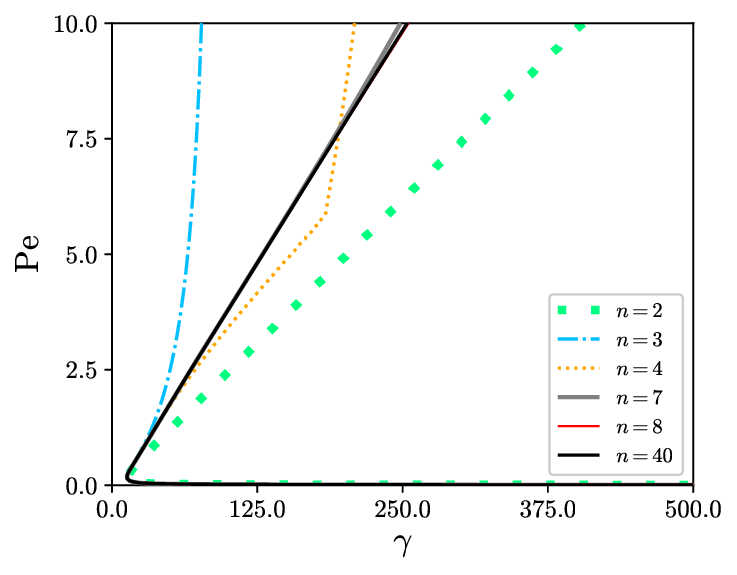}\label{fig:lambdainstabline}}
    \caption{Lines $\text{Re}(\sigma_n^\text{max}) = 0$ in the $\Pe$-$\gamma$ plane obtained from solving \eqref{eq:evproblem} truncated at $n$ for $\lambda = 0$ (a) and $\lambda = 0.1$ (b). As $n\to\infty$, the lines converge to a single line (the cases $n =2$ and $n = 8$ are indistinguishable from $n = 40$ for $\lambda = 0$ and $\lambda = 0.1$ respectively thanks to the rapid convergence of the scheme.  The linearly unstable region is to the right of the converged line. In (a), we also plot the line corresponding to the adiabatic closure \cref{adiabatic_closure}. Other parameters are $D_T=0.01,\alpha=1,\omega=2\pi$.}
    \label{fig:dispersion}
\end{figure}

The linear instability line obtained with our numerical scheme can be compared with theoretical results in the literature for the $\lambda=0$ case, as in the AAA model. In particular, given the convergence observed in \cref{fig:dispersion}(a), we consider $n=2$. Truncating the system \eqref{eq:evproblem} at $n = 2$ results in the following dispersion relation
\begin{equation} \label{M_2}
    -\tfrac{1}{2}\Pe^2+\frac{\Pe\gamma }{8\pi^2+2\alpha}- (1 + 4\pi^2D_T) D_T=0.
\end{equation}
We may compare \eqref{M_2} with the adiabatic closure method employed in \cite{pohl2014dynamic,liebchen2017phoretic}, which can be described as follows. They consider the system of equations for the moments ${\bf p}_k = \int {\bf e}(k \theta) f \dd \theta$, where $\bf{p}_0 \equiv \rho$ and $\bf{p}_1 \equiv \bf{p}$, and close the system at order $n$ by setting $f_k = 0$ for $k>n$ and $\partial_t f_n =0$. This results in a hydrodynamic model for $\rho, \bf{p},\bf{p}_2, \dots, \bf{p}_{n-1}$ and $c$. Lastly, the linear instability criterion is obtained by linearising the hydrodynamic model around the homogeneous state and dropping second-order spatial derivatives of ${\bf p}_k$ for $k \ge 1$. 
Cast in terms of our scheme, the adiabatic closure is equivalent to considering a truncated $n\times n$ system 
$$
\begin{pmatrix}
        a & b & 0 & \hdots & 0 & 0\\
        iC_1+2b & -1 & b & \hdots & 0 & 0\\
        0 & b & -4 & \hdots  & 0 & 0\\
        \vdots & \vdots & \vdots & \ddots  & \vdots & \vdots \\
        0 & 0 & 0 &\hdots  & -(n-2)^2  & b\\
        0 & 0 & 0 &\hdots  & b& -(n-1)^2 
    \end{pmatrix}\begin{pmatrix}
        A_0\\ A_1 \\ A_2 \\ \vdots \\ A_{n-2} \\ A_{n-1} 
    \end{pmatrix}=\sigma \begin{pmatrix}
        A_0\\ A_1 \\ A_2\\ \vdots \\ A_{n-2}  \\ A_{n-1}
    \end{pmatrix}\!,
$$
where $a$ and $b$ are as in \cref{params_disp}.
For $n = 2$, the adiabatic closure leads to \cite{pohl2014dynamic,liebchen2017phoretic}
\begin{equation} \label{adiabatic_closure}
    -\tfrac{1}{2}\Pe^2+\frac{\Pe\gamma }{8\pi^2+2\alpha}-D_T=0.
\end{equation}
Comparing \eqref{M_2} with \eqref{adiabatic_closure}, we see that the latter is missing a term proportional to $D_T^2$, which will be small for small $D_T$. This is what we can observe in \cref{fig:dispersion}(a), which uses $D_T = 0.01$: the adiabatic closure line \eqref{adiabatic_closure} is almost indistinguishable from the $n=2$ and the converged line obtained via our truncation scheme. However, increasing $D_T$ will lead to noticeable differences. In principle, the adiabatic closure method could be applied to the model for $\lambda>0$, but as Figure \ref{fig:lambdainstabline} indicates, it would not capture the instability line for low truncations with $n\leq 7$.

%%%%%%%%%%%%%%%%%%
\subsection{The role of the look-ahead mechanism in the instabilities of the homogeneous state}\label{sec:num}
%%%%%%%%%%%%%%%%%%
The previous sections clearly show the model behaviour in the linearly stable region. In this section, we investigate the linearly unstable region using numerical methods. 
We begin by exploring how the eigenfunction associated with the eigenvalue $\sigma_n^\text{max}$ changes with $\lambda$. Then, we present time-dependent numerical simulations of the macroscopic model \cref{eq:fcresc} showing complex emergent behaviour consistent with lane formation. Finally, we combine the simulation results with the linear theory and quantitatively classify the steady states into three different categories: (i) the homogeneous state, (ii) an aggregation spot and (iii) a lane reflecting particles moving up and down a trail of pheromone.

%%%%%%%%%%%%%%%%%%
\subsubsection{Eigenfunctions of the linearised operator}\label{sec:lineffeigf}
%%%%%%%%%%%%%%%%%%

We explore how the eigenfunction associated with the eigenvalue with the largest real part $\sigma_n^\text{max}$ changes with $\lambda$. In particular, we pick the pair $(\gamma,\Pe) = (325, 3.5)$ and $\lambda= 0, 0.1$ such that we are 
in the region where $\text{Re}(\sigma_n^\text{max})>0$ (see \cref{fig:dispersion}) and consider the eigenfunction (cf.  \cref{ansatz})
 \begin{equation}
 	\tilde f_n(x, \theta) = \text{Re}\left[\sum_{0\le k\le n}A_k\cos(k\theta) \exp(\ii\omega x)\right]
\end{equation}
of the truncated eigenvalue problem. The eigenfunctions $\tilde f_n(x, \theta)$ with $n=40$ for the cases $\lambda= 0$ and $\lambda= 0.1$  are shown in
\cref{fig:eigflambda0,fig:eigflambda1} respectively. 

Without the look-ahead mechanism ($\lambda = 0$), the eigenfunction has peaks at $\theta = 0$ and $\theta = \pi$ slightly offset, left and right, respectively, of $x = 0$. This corresponds to ants facing left and right and balancing each other to form an immobile lane or a ``one-dimensional spot'' (see \cref{fig:lanelambda0}). When the look-ahead mechanism is turned on ($\lambda = 0.1$), the peaks shift to angles $\frac{1}{2}\pi$ and $-\frac{1}{2}\pi$ as in Figure \ref{fig:eigflambda1}. This reflects most of the particles moving up and down along the middle of a lane, as illustrated in Figure \ref{fig:lanelambda1}. 
In conclusion, the dominating patterns are different for $\lambda=0$ and $\lambda>0$, and turning on the look-ahead mechanism gives rise to lane formation at the linear level.
\begin{figure}[htb]    
\centering
    \subfloat[$\lambda=0$]{
         \includegraphics[width=0.48\linewidth]{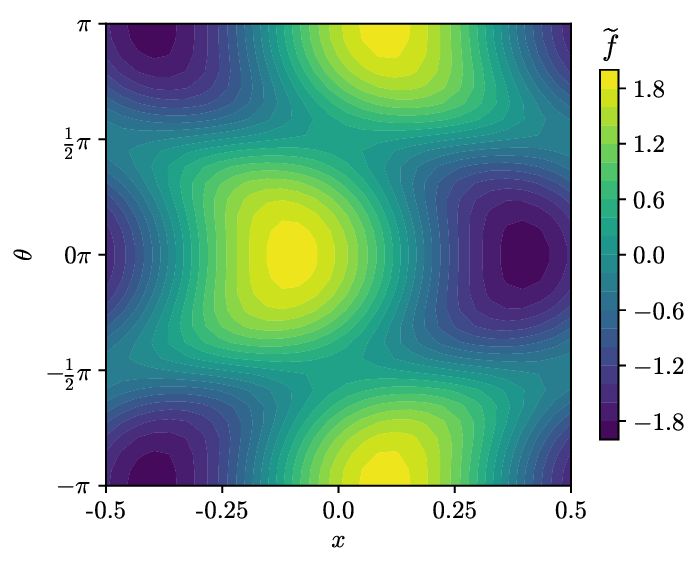}
         \label{fig:eigflambda0}}
    \subfloat[$\lambda=0.1$]{
         \includegraphics[width=0.48\linewidth]{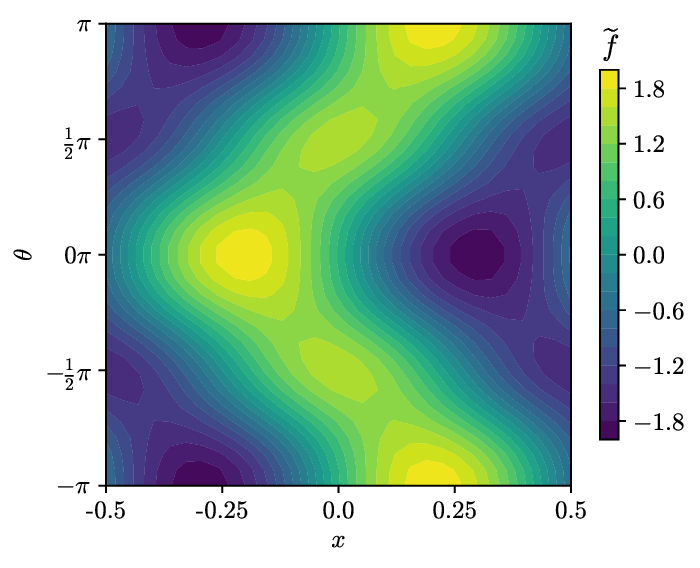}
         \label{fig:eigflambda1}}
    \caption{Illustration of the effect of introducing the look-ahead mechanism to the eigenfunctions of the linear operator. Plotted are heatmaps for the value of the eigenfunction $\tilde{f}(x,\theta)$ of the truncated matrix with the largest real part. Parameters: $D_T=0.01,\Pe=3.5,\gamma=325.0,n=40$.}
    \label{fig:eigf_effect_lambda}
\end{figure}
\begin{figure}[htb]
    \centering
    \subfloat[$\lambda=0$]{
         \includegraphics[width=0.48\linewidth]{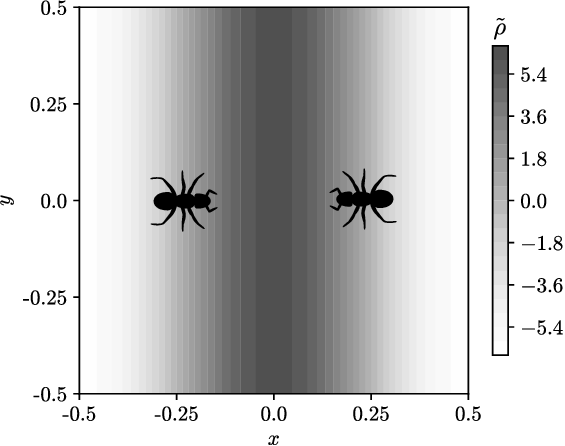}
         \label{fig:lanelambda0}}
    \subfloat[$\lambda=0.1$]{
         \includegraphics[width=0.48\linewidth]{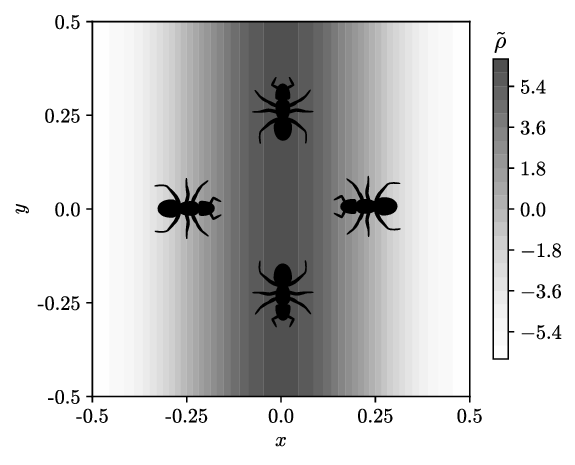}
         \label{fig:lanelambda1}}
    \caption{Illustration of the effect of introducing the look-ahead mechanism to the eigenfunctions of the linear operator. The peaks in Figures \ref{fig:eigflambda0} and \ref{fig:eigflambda1} are illustrated here as ants at the locations of the peak with orientations in the direction of the peak in the accompanying plot \ref{fig:eigflambda0} and \ref{fig:eigflambda1}, respectively. The background reflects the spatial density associated with the eigenfunction $\tilde{f}$, $\tilde{\rho}=\int_0^{2\pi}\tilde{f}\dd\theta$, to indicate what the lane looks like spatially.}
    \label{fig:eigf_effect_lambda2}
\end{figure}

%%%%%%%%%%%%%%%%%%
\subsubsection{Time-dependent simulations} \label{sec:pde_scheme}
%%%%%%%%%%%%%%%%%%

We approximate solutions to the macroscopic model \cref{eq:fcresc} using a finite volume scheme similar to \cite{bruna2022phase}. The scheme can be derived by using the mobility form of equation (\ref{eq:fresc}): 
\begin{equation}
    \partial_t f + \nabla_\xi\cdot(f\bf{U})= 0,\qquad {\bf U } =  \begin{pmatrix} -D_T\partial_x \log f+\Pe\cos\theta\\ -D_T\partial_y\log f+\Pe\sin\theta\\-\partial_\theta \log f+\gamma\bf{n}_\theta\cdot\nabla c_\lambda \end{pmatrix}
\end{equation} where $\bf{U} = (U^x, U^y, U^\theta)$ the velocity vector.

The domain $\Sigma=\T^2\times \T_{2\pi}$ is discretised into $N_x\times N_y\times N_\theta$ cells indexed by $(i,j,k)$ 
\begin{equation}
    C_{i,j,k}=[(i-1)\Delta x,i\Delta x]\times[(j-1)\Delta y,j\Delta y]\times[(k-1)\Delta \theta,k\Delta \theta],
\end{equation} where $\Delta x=1/N_x,\Delta y=1/N_y$ and $\Delta\theta=2\pi/N_\theta$. Then we approximate $f(t, x_i, y_j, \theta_k)$ with $x_i = i\Delta x, y_j = j \Delta y, \theta_k = k \Delta \theta$ by the cell--averages 
\begin{equation}f_{i,j,k}(t)=\frac{1}{\Delta x\Delta y\Delta \theta}\iiint_{C_{i,j,k}}f(x,y,\theta,t)\dd x\dd y\dd \theta.
\end{equation}
We use the finite-volume scheme 
\begin{equation}
\frac{\dd}{\dd t}f_{i,j,k}=-\frac{F_{i+1/2,j,k}^x-F_{i-1/2,j,k}^x}{\Delta x}-\frac{F_{i,j+1/2,k}^y-F_{i,j-1/2,k}^y}{\Delta y}-\frac{F_{i,j,k+1/2}^\theta-F_{i,j,k-1/2}^\theta}{\Delta \theta},
\end{equation} 
for $i = 0, \dots, N_x-1, j = 0, \dots N_y-1, k = 0, \dots, N_\theta - 1$. We approximate the flux $F^x$ 
at the cell interfaces by the numerical upwind flux
\begin{equation}
	F^x_{i+1/2,j,k} = (U_{i+1/2,j,k}^x)^+ f_{i,j,k} + (U_{i+1/2,j,k}^x)^- f_{i+1,j,k}, 
\end{equation}
using $(\cdot)^+ = \max(\cdot, 0)$ and $(\cdot)^- = \min(\cdot, 0)$, and similarly for $F^y$ and $F^\theta$. The velocities $U^x, U^y, U^\theta$ are approximated by centered differences, e.g., the $x$-velocity is
\begin{equation}
U_{i+1/2,j,k}^x=-D_T\frac{\log f_{i+1,j,k}-\log f_{i,j,k}}{\Delta x}+ \Pe \cos\theta_k.
\end{equation}
The angular velocity is discretised differently depending on $\lambda$. For $\lambda = 0$, we use
\begin{equation}
\begin{aligned}
U_{i,j,k+1/2}^\theta=&-D_R\frac{\log f_{i,j,k+1}-\log f_{i,j,k}}{\Delta \theta}-\gamma\sin( \theta_{k+1/2})[\partial_x c_\lambda]_{i,j}
\\&+\gamma\cos(\theta_{k+1/2})[\partial_y c_\lambda]_{i,j},
\end{aligned}\end{equation}
with
\begin{align}
    [\partial_x c]_{i,j}=\frac{c_{i+1,j}-c_{i-1,j}}{2\Delta x}, \qquad 
    [\partial_y c]_{i,j}=\frac{c_{i,j+1}-c_{i,j-1}}{2\Delta y}.
\end{align}
For $\lambda>0$, we use the identity: 
\begin{equation}
    \bf{n}_\theta\cdot\nabla c(\xx+\lambda\bf{e}_\theta)=\frac{1}{\lambda}\partial_\theta(c(\xx+\lambda\bf{e}_\theta)),
\end{equation}
so that the $\theta$-flux becomes 
\begin{equation}
    U_{i,j,k+1/2}^\theta=-D_R\frac{\log f_{i,j,k+1}-\log f_{i,j,k}}{\Delta \theta}+\frac{\gamma}{\lambda}\frac{c(\xx_{i,j}+\lambda\bf{e}_{\theta_{k+1}})-c(\xx_{i,j}+\lambda\bf{e}_{\theta_k})}{\Delta\theta}.
\end{equation}
We solve the chemical field equation \eqref{eq:cresc}, $\alpha c - \Delta c = \rho$, using second-order finite differences with the right-hand side computed by
\begin{equation}
	\label{rho_discrete}
	\rho_{i,j}(t) = \Delta \theta \sum_{k=1}^{N_\theta} f_{i,j,k}(t).
\end{equation}
The shifted chemical field $c_\lambda$ requires evaluations at  the points $\xx_{i,j}+\lambda\bf{e}_{\theta_k}$. We linearly interpolate the values $c_{i,j}$ using the  \texttt{Interpolations} package in \texttt{Julia}. Lastly, the resulting system of ODEs for $f_{i,j,k}(t)$ is solved using a forward Euler method with adaptive time-stepping. We used the same time-stepping condition as in \cite{bruna2022phase}.

%%%%%%%%%%%%%%%%%%%%%%%
\subsubsection{Spots and lanes}
%%%%%%%%%%%%%%%%%%%%%%%

\cref{fig:spot_rho_p,fig:stripe_rho_p} show examples of the typical behaviour of the macroscopic model \cref{eq:fcresc} obtained with the numerical scheme of \cref{sec:pde_scheme}. They show the evolution of the spatial density $\rho(t,\xx)$ \cref{eq:rho} and the polarisation ${\bf p}(t,\xx)$ \cref{eq:p} at four time instances. The first time corresponds to the initial condition, which is uniformly random (in space and orientation) and normalised such that $\int f\dd\xx\dd\theta=1$, and the last one corresponds to an already equilibrated or steady-state solution.
The values of the polarisation at the grid locations are obtained analogously to those of $\rho_{i,j}$ (see \cref{rho_discrete}). 

\cref{fig:spot_rho_p} shows the typical cluster formation or aggregation for small $\Pe$, 
already present in model \cref{eq:fcresc} with $\lambda = 0$ (or the AAA model \cite{liebchen2017phoretic,pohl2014dynamic}), and similar to the Keller--Segel collapse. Hence, it corresponds to the macroscopic view of the particle behaviour shown in \cref{fig:curvelambda0}. We call these endstates \emph{spots}. 

When the look-ahead mechanism is turned on ($\lambda > 0$) and $\Pe$ is sufficiently large, we observe the formation of a \emph{lane}, representative of trail formation (\cref{fig:stripe_rho_p}). This solution corresponds to a lane with bidirectional movement along the stripe, although it is difficult to see it in \cref{fig:stripe_rho_p}. To this end, we plot the solution $f(t,\xx,\theta)$ at the final time in the $(x,\theta)$ domain in \cref{fig:fytheta_plot}. This shows two peaks at $x \approx -0.1$, the location of the centre of the lane, and $\theta =\frac{1}{2}\pi, -\frac{1}{2}\pi$, meaning that particles predominantly move up and down. Moreover, slightly left to the lane, $x \lessapprox -0.1$, the angles shift to the left semi-plane, meaning that ants turn slightly to join the lane. Analogously, for $x \gtrapprox -0.1$, the angles turn towards $\theta = \pm\pi$, so ants move back towards the lane.

We note the similarity of \cref{fig:fytheta_plot} to the eigenfunction associated with the linear instability in \cref{fig:eigflambda1}.  This shows that the linear stability theory of \cref{sec:lineffeigf} is consistent with the patterns that emerge at the nonlinear level. In Appendix \ref{sec:FEM}, we show further evidence that model \cref{eq:fcresc} admits lanes as stationary solutions. The lane endstate of \cref{fig:stripe_rho_p} is a macroscopic representation of the particle behaviour of \cref{fig:curvelambda1}, in the sense that particles move along a line in both scenarios.
\begin{figure}[htb]
    \centering
   \includegraphics[width=0.9\textwidth]{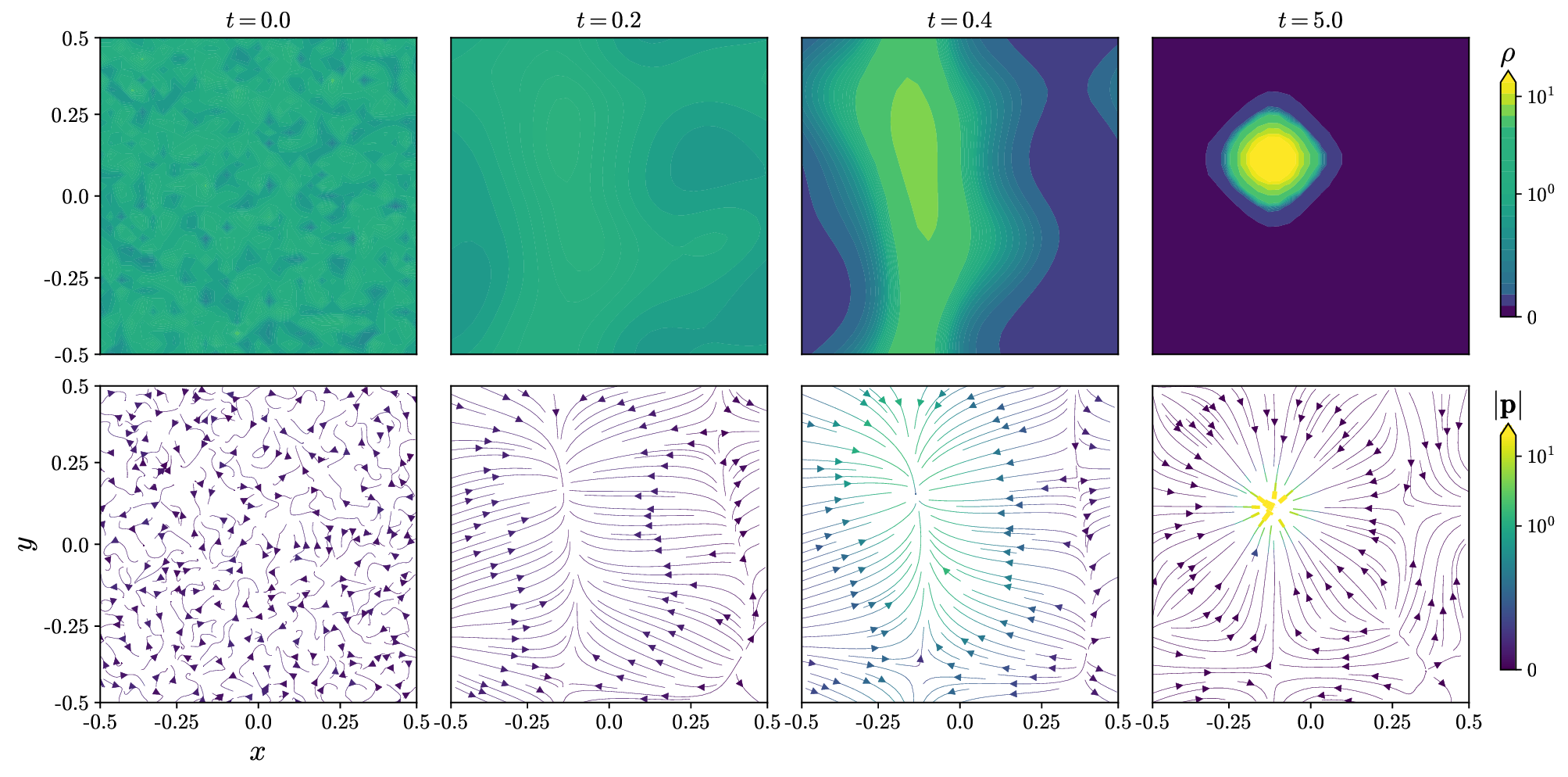}
    \caption{Example of typical behaviour for the model. The initial condition is randomly uniform on $[0,1]$ for each $f_{i,j,k}(t=0)$. The endstate in this simulation represents a steady cluster. Parameters: $D_T=0.01,\Pe=1.5,\gamma=325,\lambda=0.1,N_x=N_y=31,N_\theta=21,\Delta t=10^{-5}$.}
    \label{fig:spot_rho_p}
\end{figure}
\begin{figure}[htb]
    \centering
   \includegraphics[width=0.9\textwidth]{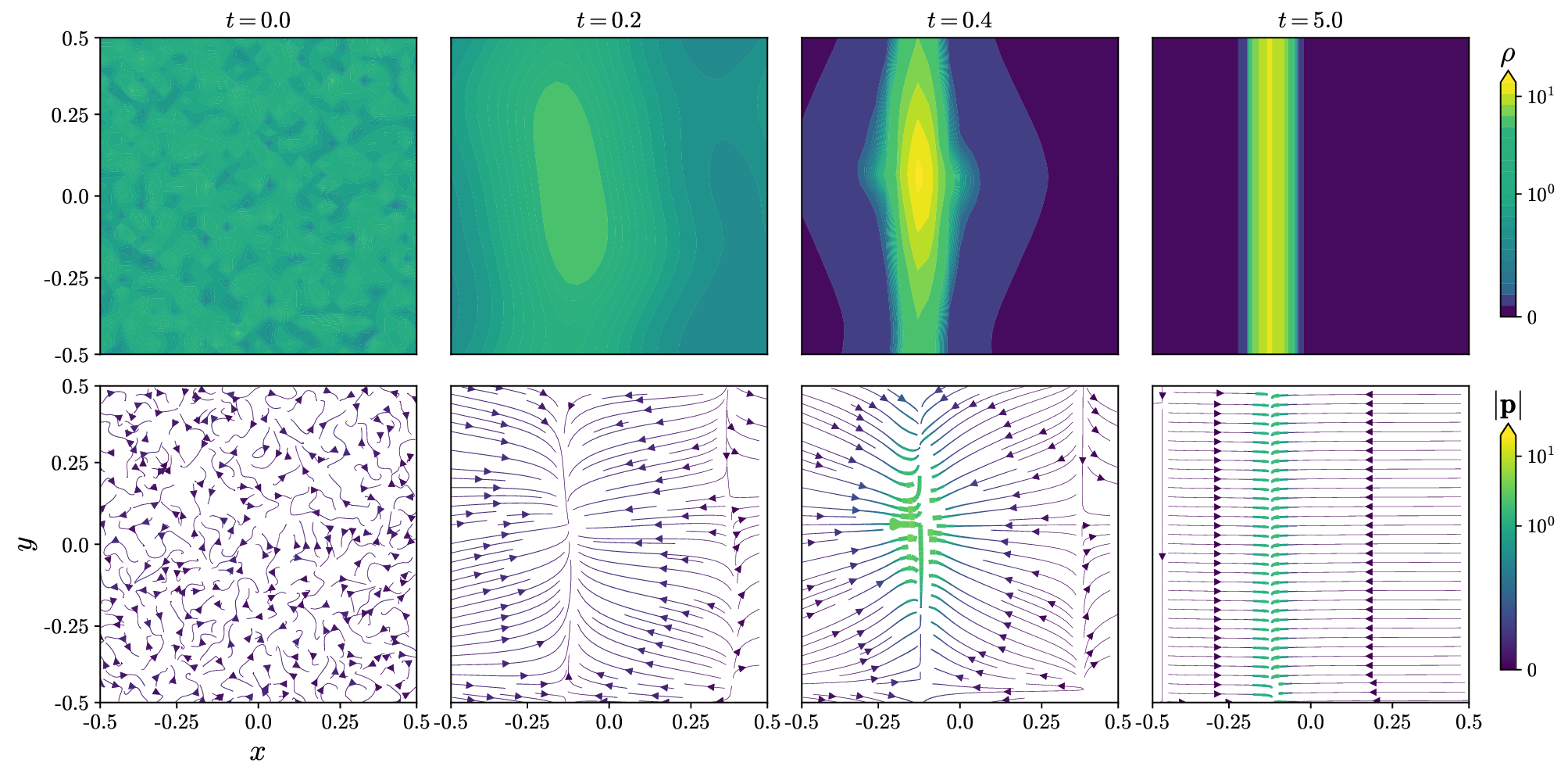}
    \caption{Example of typical behaviour for the model. The initial condition is randomly uniform on $[0,1]$ for each $f_{i,j,k}(t=0)$. The endstate in this simulation represents a steady lane. Parameters: $D_T=0.01,\Pe=3.5,\gamma=325,\lambda=0.1,N_x=N_y=31,N_\theta=21,\Delta t=10^{-5}$.}
    \label{fig:stripe_rho_p}
\end{figure}
\begin{figure}[htb]
   \centering\includegraphics[width=0.5\textwidth]{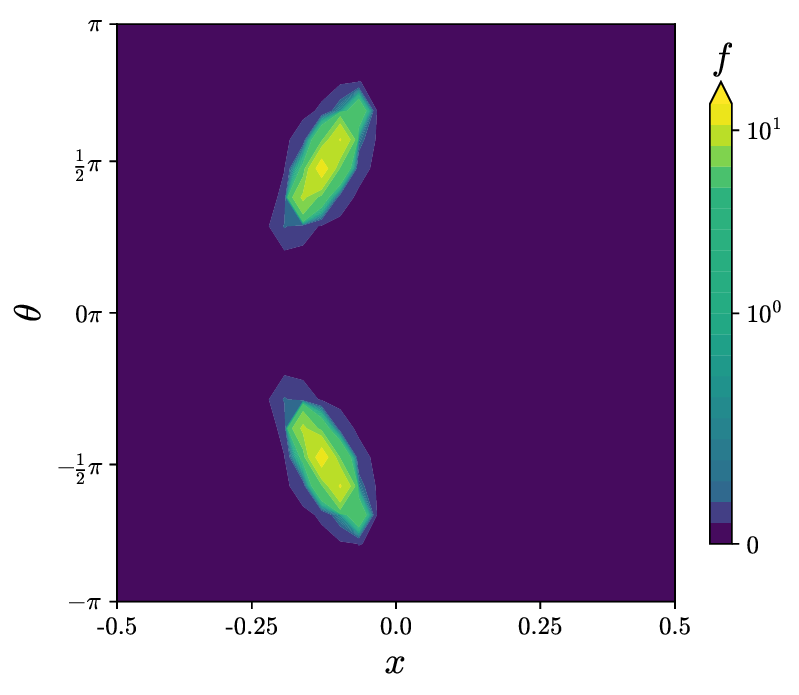}
    \caption{Heatmap of $f(t=5.0,x,y,\theta)$ from Figure \ref{fig:stripe_rho_p} showing two peaks in the polarsation. Changing $y$ does not change the plot, hence the heatmap looks the same for all other $y$. Parameters: $D_T=0.01,\Pe=3.5,\gamma=325,\lambda=0.1,N_x=N_y=31,N_\theta=21,\Delta t=10^{-5}$.}
    \label{fig:fytheta_plot}
\end{figure}

Our simulations suggest that model \cref{eq:fcresc} with $\lambda > 0$ is bistable for some parameter values. In particular, in the region of linear instability of the homogeneous state (see \cref{fig:lambdainstabline}), we observe some pairs $(\gamma, \Pe)$ that result in either a spot or a lane endstate depending on the initial condition. This is shown in \cref{fig:bistable} for a pair $(\gamma, \Pe)$ in between those used in \cref{fig:spot_rho_p,fig:stripe_rho_p}. It shows two realisations of the solution at $t = 5$ given a uniformly random initial condition. That is, for a fixed $\gamma$, there is a range of $\Pe$ that results in the model having two basins of attraction, one corresponding to spot endstates and one related to lane endstates. For low $\Pe$, the dominant basin is that of spots (\cref{fig:spot_rho_p}) while for moderate $\Pe$ (below the homogeneous state stability line), the dominant basin is that of lanes (\cref{fig:stripe_rho_p}).

\begin{figure}[htb]
    \centering
   \subfloat[]{\includegraphics[width=0.45\textwidth]{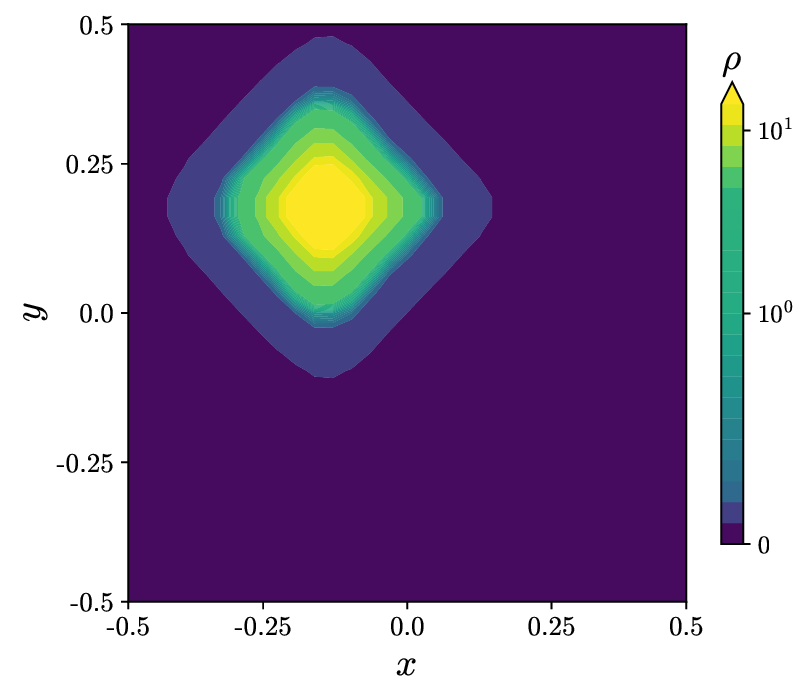}\label{fig:spot_bistab}}
   \subfloat[]{\includegraphics[width=0.45\textwidth]{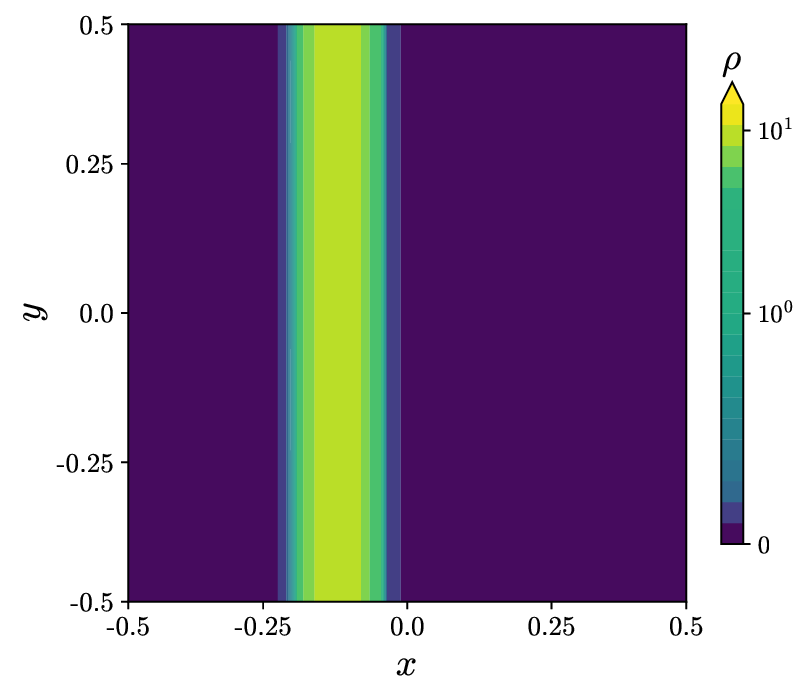}\label{fig:lane_bistab}}
    \caption{Two different realisations of simulations for the same parameters but with different initial conditions. Heatmaps for $\rho(t=5.0,\xx)$ are shown at the final time. The seeding for the random uniform initial condition was chosen differently in each case. Other parameters: $D_T=0.01,\Pe=2.5,\gamma=325,\lambda=0.1,N_x=N_y=31,N_\theta=21,\Delta t=10^{-5}$.}
    \label{fig:bistable}
\end{figure}
\begin{figure}[htb]
    \centering
   \includegraphics[width=0.6\textwidth]{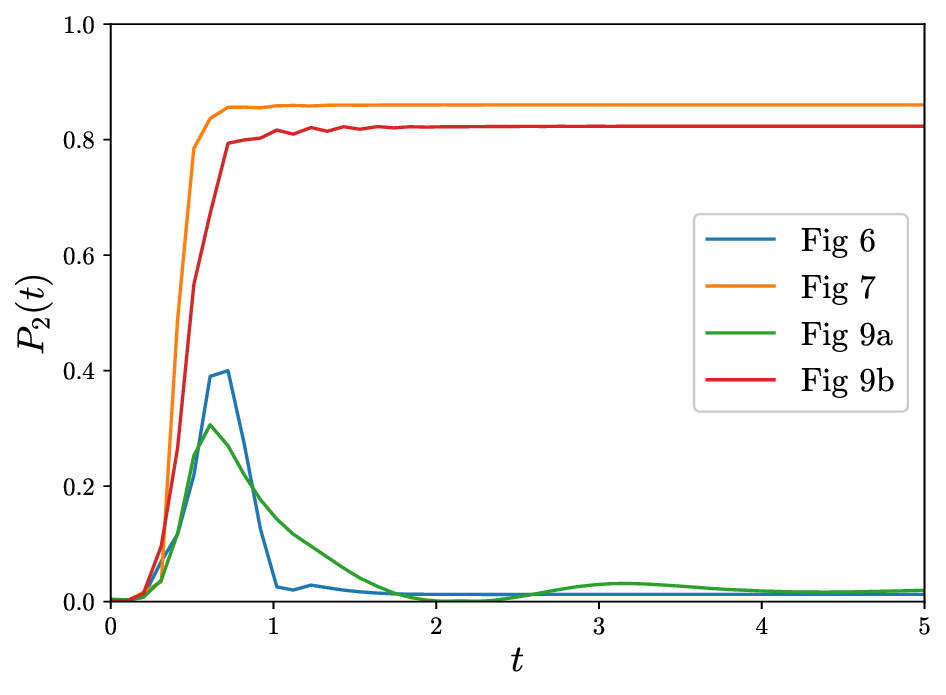}
    \caption{The evolution of $P_2$ as for the data from Figures \ref{fig:spot_rho_p}, \ref{fig:stripe_rho_p}, \ref{fig:spot_bistab} and \ref{fig:lane_bistab} respectively. The seeding for the random uniform initial condition was chosen differently in each case. Other parameters: $D_T=0.01,\gamma=325,\lambda=0.1,N_x=N_y=31,N_\theta=21,\Delta t=10^{-5}$.}
    \label{fig:P2_evol}
\end{figure}

%%%%%%%%%%%%%%%%%%%%%%%%
\subsubsection{Quantitative classification of spots versus stripes}
%%%%%%%%%%%%%%%%%%%%%%

We have seen that, for certain parameter values, the solutions of the macroscopic model \cref{eq:fcresc} with $\lambda > 0$ converge to either a spot or a lane. We now present a quantitative method to discern lanes from spots based on the mean second-order moment 
\begin{equation}\label{eq:bigP2}
    P_2(t)=\left |\int_{\T^2}\bf{p}_2(t,\xx)\dd\xx \right|= \left| \int_{\T^2}\int_0^{2\pi}\bf{e}(2\theta)f(t,\xx,\theta)\dd\xx\dd\theta \right|.
\end{equation} 
The choice for $P_2$ is motivated by the shape of the eigenfunctions as discussed in \cref{fig:eigf_effect_lambda}. The form of the eigenfunctions is \begin{equation}
    \sum_{k\geq 0}A_k \cos(k\theta),
\end{equation} and we saw that the eigenfunctions predicted the patterns well that emerge at the nonlinear level. Also, the coefficients $(A_k)_{k\geq 0}$ typically decay, so the first terms dominate the shape of the pattern. If we only include $k=0$ or $k=0,1$, we can only get homogeneous solutions or one-dimensional spots as illustrated in \cref{fig:lanelambda0}, respectively. The first relevant coefficient for lane formation will be $A_2$ for $k=2$. Projecting onto the coefficient $A_2$ is exactly what the mean second-order moment $\bf{p}_2(t,\xx)$ represents, and hence, integrating over the whole spatial domain gives a measure for how much of the total density is contributing to the lane formation. In \cref{fig:p2field} we plot the associated vectorfield of $\bf{p}_2(t,\xx)$ for the simulations from \cref{fig:spot_rho_p} and \cref{fig:stripe_rho_p}. This shows that $\bf{p}_2(t,\xx)$ is relatively large and unidirectional in the spatial region where the lane is located and sums to zero in the case of a spot.
\begin{figure}[htb]
    \centering
    \subfloat[Fig \ref{fig:spot_rho_p}]{\includegraphics[width=0.4\textwidth]{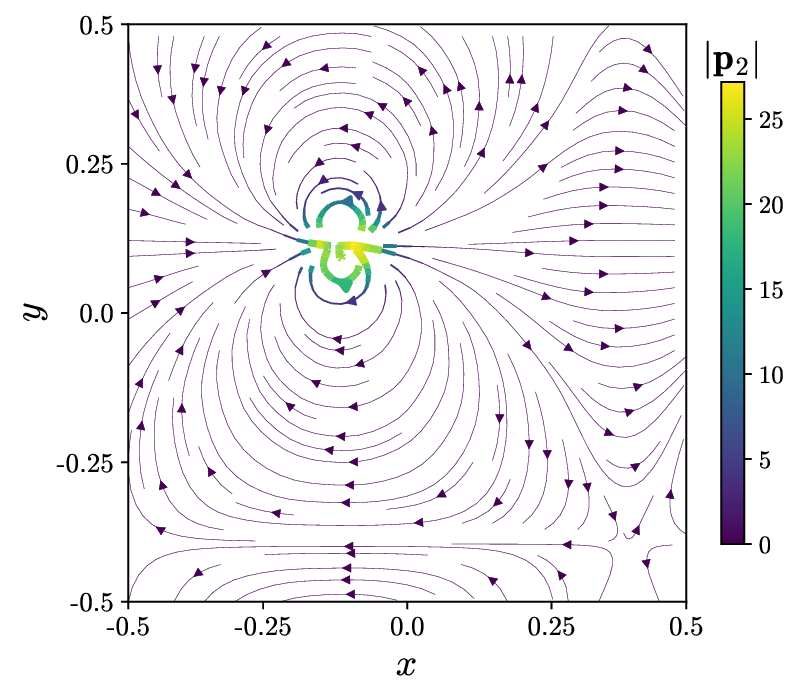}}
    \subfloat[Fig \ref{fig:stripe_rho_p}]{\includegraphics[width=0.4\textwidth]{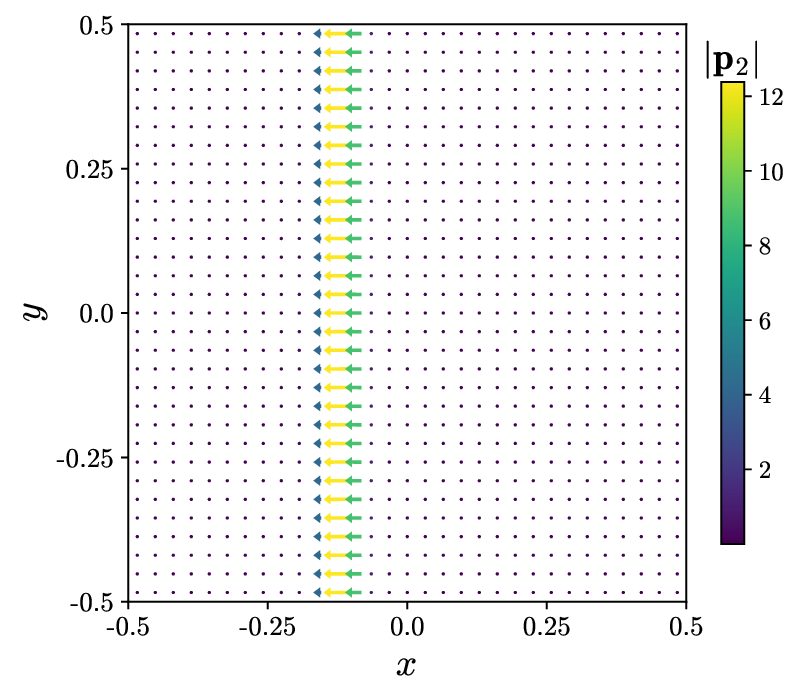}}
    \caption{Vectorfield plots for $\bf{p}_2(t=5.0,\xx)$ at the final time with the simulation data from (a) Figure \ref{fig:spot_rho_p} and (b) Figure \ref{fig:stripe_rho_p}. Other parameters: $D_T=0.01,\gamma=325,\lambda=0.1,N_x=N_y=31,N_\theta=21,\Delta t=10^{-5}$.}
    \label{fig:p2field}
\end{figure}

We evaluate $P_2(t)$ numerically from the PDE simulations using straightforward discrete integrals (cf. \cref{rho_discrete}). Lane solutions are characterised by large values of $P_2$, while spot solutions (as well as, obviously, the homogeneous solution) have low $P_2$.  In \cref{fig:P2_evol}, the time evolution of $P_2(t)$ is plotted for the simulations as shown in \cref{fig:spot_rho_p,fig:stripe_rho_p,fig:spot_bistab,fig:lane_bistab}, showing the distinctive behaviour of $P_2$ for either a lane or a spot.

Next, we run simulations of \cref{eq:fcresc} for a range of parameter pairs $(\gamma, \Pe)$  up to time $t = 5$ (by which time all solutions have equilibrated). For each pair $(\gamma, \Pe)$, we consider eight realisations using uniformly random initial data and evaluate the solutions at the final time $t = 5$ (see \cref{fig:rho706,fig:rho1001,fig:rho4472,fig:rho5555,fig:rho6061,fig:rho8154,fig:rho9437,fig:rho9956} in Appendix D).
In particular, we compute the norm of the distance to the homogeneous solution and keep the largest of the eight, and the average of the mean second moment $P_2$ \cref{eq:bigP2},
\begin{equation}
	d_{f_\ast}\equiv \max_j \|f(t_\text{max}) - f_\ast \|_{L^2}, \qquad \bar{P_2}\equiv\langle |P_2(t_\text{max})| \rangle_j,
\end{equation}
where $j = 1, 2, \dots, 8$ labels the trajectory. \cref{fig:l2p2} shows the values of these two quantities in the $(\gamma, \Pe)$ plane. The black line is the dispersion relation from the linear stability theory from \cref{fig:lambdainstabline}. 
\begin{figure}[htb]
\centering
   \subfloat[$d_{f_\ast}$]{\includegraphics[width=0.48\textwidth]{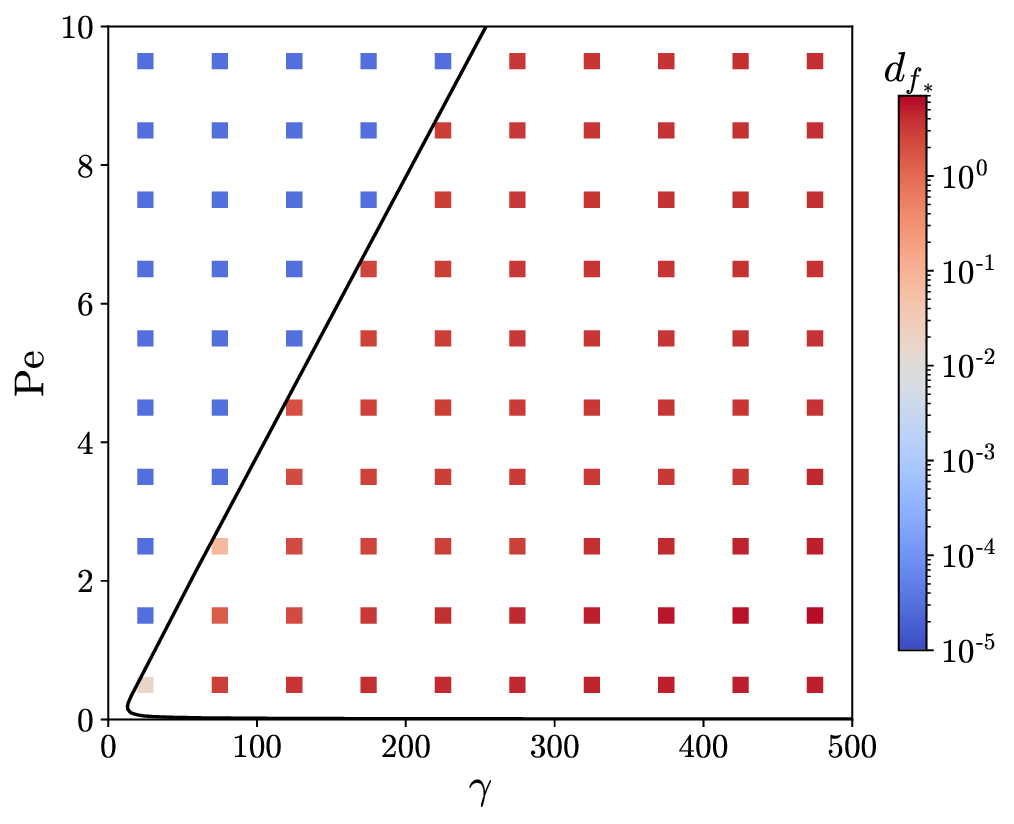}\label{fig:l2max}}
   \subfloat[$\bar{P_2}$]{  \includegraphics[width=0.48\textwidth]{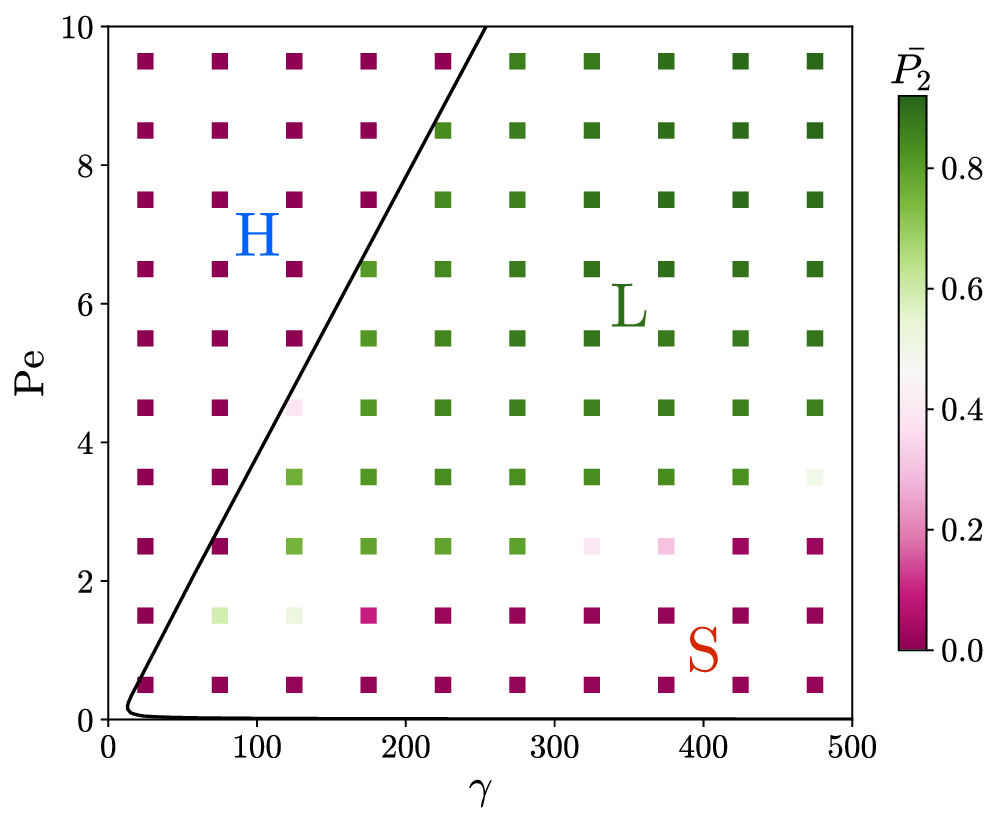}\label{fig:meanP2}}
\caption{Scatter plot for different seedings for 8 simulations for each pair $\Pe$--$\gamma$. The initial condition depends on the seeding and is chosen to be randomly uniform on $[0,1]$ for each $f_{i,j,k}(t=0)$. In Figure \ref{fig:l2max}, the maximum value of the discretised value of $\|f-f_\ast\|_{L^2}$ at the end $t=5.0$ for these eight simulations is plotted. In Figure \ref{fig:meanP2} the mean value of the discretised value of $P_2$ at the end $t=5.0$ for the 8 simulations is plotted, where $P_2$ is defined as in (\ref{eq:bigP2}). The black line in both figures is the linear instability line from \cref{fig:lambdainstabline}. The letters in Figure \ref{fig:meanP2} are meant to make clear that in the region to the left of the black line, the solutions converge to the homogeneous state (H) and right to it predominantly to a lane (L) or a spot (S), with a bistable region in between the L and S region. In Appendix D, heatmaps of $\rho(t=5.0,\xx)$ are shown for each different seeding; see \cref{fig:rho706,fig:rho1001,fig:rho4472,fig:rho5555,fig:rho6061,fig:rho8154,fig:rho9437,fig:rho9956}. Other parameters are $D_T=0.01,\lambda=0.1,\alpha=1,N_x=N_y=31,N_\theta=21,\Delta t=10^{-5}$.}\label{fig:l2p2}
\end{figure}

The distance to the homogeneous solution (\cref{fig:l2max}) shows excellent agreement between the linear stability theory and the nonlinear stability from simulations of the original model. That is, left to the linear instability line, the simulations of the nonlinear model converge to the homogeneous state as the $L^2$ norm of the difference $f-f_\ast$ is small at $t=5.0$ (this can also be clearly seen in the spatial density plots in \cref{fig:rho706,fig:rho1001,fig:rho4472,fig:rho5555,fig:rho6061,fig:rho8154,fig:rho9437,fig:rho9956} in Appendix D). This region is denoted by the letter `H' in Figure \ref{fig:meanP2}. Right to the linear instability line, we see that the solutions do not converge to the homogeneous state (the $L^2$ norm of the difference $f-f_\ast$ is not small at $t=5.0$). 

While the norm of the distance to the homogeneous solution $f_\ast$ allows us to identify non-trivial stationary solutions, with that alone, we cannot distinguish between spots and lanes. \cref{fig:meanP2} shows that the average of $P_2$ subdivides the region of linear instability into two regions, labelled by `S' and `L'. The `S' region has low $P_2$ and corresponds to the spot-dominant region; that is, all eight realisations converged to spot-like solutions (see \cref{fig:rho706,fig:rho1001,fig:rho4472,fig:rho5555,fig:rho6061,fig:rho8154,fig:rho9437,fig:rho9956} in Appendix D). Instead, the `L' region has $P_2$ close to unity as characteristic of lane-like solutions. The thin boundary region between the `S' and `L' regions is the region of bistability between spot and lane solutions. There, the average $P_2$ takes intermediate values since a proportion of the realisations converged to spots while others converged to lanes (see \cref{fig:rho706,fig:rho1001,fig:rho4472,fig:rho5555,fig:rho6061,fig:rho8154,fig:rho9437,fig:rho9956} in Appendix D). 

In conclusion, the linear theory gives a good prediction for the nonlinear stable region, the value of $P_2$ is an excellent measure to distinguish between spots and lanes, and there is a region of bistability in the $\Pe$-$\gamma$ plane.

%%%%%%%%%%%%%%%%%%%%%%%%%%%%%%%
\appendix
%%%%%%%%%%%%%%%%%%%%%%%%%%%%%%%

%%%%%%%%%%%%%%%%%%%%%%%%%%%%%%%
\section{Numerical scheme for particle simulations}\label{sec:particlenumscheme}
%%%%%%%%%%%%%%%%%%%%%%%%%%%%%%%

To simulate the microscopic model \cref{sde_model}, 
we discretise the SDEs for the positions ${\bf X}_i$ and the orientations $\Theta_i$ with the tamed Euler scheme \cite{hutzenthaler2012strong}
\begin{subequations}
    \begin{align}
    \XX_i(t + \Delta t) &=\XX_i(t ) +  v_0\bf{e}(\Theta_i)\Delta t+\sqrt{2D_T\Delta t}\, \boldsymbol{\zeta}_i,\\
     \Theta_i (t + \Delta t)&= \Theta_i(t) + \frac{F_i \Delta t}{1+|F_i|\Delta t}+\sqrt{2D_R\Delta t}\,\zeta_i, \label{discrete_theta}
\end{align}
\end{subequations}
where $\boldsymbol{\zeta}_i$ and $\zeta_i$ are 2D and 1D standard normal $N(0,1)$ random vectors, respectively. The drift term in \cref{discrete_theta} is
\begin{equation}\label{Fi}
    F_i=\frac{\gamma}{N} \bf{n}(\Theta_i)\cdot \sum_{j=1}^N  \nabla K_{\alpha,D}(\XX_i + \lambda\bf{e}(\Theta_i) - \XX_j(t)),
\end{equation}
using the expression for the chemical field \cref{eq:K0}. If $\lambda = 0$, we modify $F_i$ to remove the self-interactions, $j \ne i$. 
The tamed scheme allows us to better resolve the singularities in $c_N$ without having to take prohibitively small time steps $\Delta t$. 
Periodic boundary conditions for ${\bf X}_i$ on $\T_L^2$ and for $\Theta_i $ on $\T_{2\pi}$ are imposed.

%%%%%%%%%%%%%%%%%%%%
\section{Formal mean-field limit}\label{mflimit}
%%%%%%%%%%%%%%%%%%%%

Assume that the look-ahead parameter is $\lambda >0$. 
The joint law $f^N$ of the diffusion process \cref{sde_model} for $N$ particles satisfies the $N$-particle Fokker--Planck or Liouville equation \cite{erban2019stochastic}
\begin{equation}\label{eq:liouville}
    \partial_t f^N=\sum_{i=1}^N\nabla_{\xx_i}\cdot[D_T\nabla_{\xx_i}f^N- v_0\bf{e}_{\theta_i}f^N]+\partial_{\theta_i}[D_R\partial_{\theta_i} f^N- F_i(\vec \xi) f^N],
\end{equation}
where the interaction term $F_i(\vec \xi)$ is given in \cref{Fi} , $\vec \xi = (\xi_1, \dots, \xi_N)$ and $\xi_i = (\xx_i, \theta_i)$.

From $f^N$, we can define the $k$-marginals or the $k$-particle probability density as follows 
\begin{equation} \label{marginal}
    f^N_k(t,\xi_1,\dots,\xi_k)=\int_{\Sigma^k}f^N(t,\xi_1,\dots,\xi_N)\dd\xi_{k+1}\dots \dd\xi_N,
\end{equation}
where $\Sigma=\T_L^2\times\T_{2\pi}$. We may write an equation for the first marginal $f^N_1$ by integrating \cref{eq:liouville}
\begin{equation} \label{integratedFPN}
    \begin{aligned}
    \partial_t f^N_1 &=\int_{\Sigma^{N-1}}\partial_t f^N \dd\xi_{2}\dots \dd\xi_N,\\
    &= \nabla_{\xx_1} \cdot[D_T\nabla_{\xx_1} f_1^N-v_0\bf{e}_{\theta_1} f_1^N] + D_R\partial^2_{\theta_1} f_1^N  -    
    \int_{\Sigma^{N-1}} \partial_{\theta_1} \left[ F_1(\vec \xi) f^N \right] \dd\xi_{2}\dots \dd\xi_N,
\end{aligned}
\end{equation}
where the terms with derivatives with respect to $\xx_i$ and $\theta_i$ for $i\ne 1$ vanish using the divergence theorem and the periodic boundary conditions, and 
\begin{align} \label{F1_exp}
\begin{aligned}
	F_1(\vec \xi) &= \frac{\gamma}{N} \bf{n}(\theta_1)\cdot
\sum_{i=1}^N \nabla K_{\alpha,D}(\xx_1 + \lambda\bf{e}(\theta_1) -\xx_i )\\
& = \frac{\gamma}{N} \bf{n}(\theta_1)\cdot \left[\nabla K_{\alpha,D} (\lambda\bf{e}(\theta_1) ) + 
\sum_{i=2}^N \nabla K_{\alpha,D} (\xx_1 + \lambda\bf{e}(\theta_1) -\xx_i )\right ]\\
& = \frac{\gamma}{N} \bf{n}(\theta_1)\cdot  
\sum_{i=2}^N \nabla K_{\alpha,D} (\xx_1 + \lambda\bf{e}(\theta_1) -\xx_i ).
\end{aligned}
\end{align}
Between the second and third lines, we have used that $K_{\alpha, D}(\xx) = K_0(\sqrt{\alpha/D} |\xx|)$ (see \cref{eq:K0}) and hence $\nabla K_{\alpha,D} (\lambda\bf{e}(\theta_1) ) = K_0'(\sqrt{\alpha/D} \lambda) {\bf e}(\theta_1)$, which is orthogonal to ${\bf n}(\theta_1)$. 
Thus we have that the form of $F_1$ reduces to that of the case $\lambda = 0$ where we remove the self-interactions and we may proceed with the derivation as applicable for both $\lambda = 0$ and $\lambda > 0$ cases. 

Next, we insert \cref{F1_exp} into the last term in \eqref{integratedFPN} and use that all particles are identical, which allows us to relabel particles for $i = 3, \dots, N$ and write
\begin{multline*}
	\int_{\Sigma^{N-1}} \partial_{\theta_1} \left[ F_1(\vec \xi) f^N \right] \dd\xi_{2}\dots \dd\xi_N\\
	= \frac{\gamma (N-1)}{N} \partial_{\theta_1}  \int_\Sigma \bf{n}(\theta_1)\cdot
\nabla K_{\alpha,D}  (\xx_1 + \lambda\bf{e}(\theta_1) -\xx_2 )f_2^N(t, \xi_1, \xi_2) \dd \xi_2,
\end{multline*}
where $f_2^N$ is the two-particle marginal given by \cref{marginal}. Next we assume that, as $N\to \infty$, there is propagation of chaos, namely, 
$f^N_2\rightharpoonup (f_1^\infty)^{\otimes 2}$ as $N\to\infty$ holds for some limiting function $f_1^\infty$ \cite[Definition 3.ii]{jabin2017mean}. Inserting this approximation in \eqref{integratedFPN} and writing $f_1^\infty \equiv f$ we arrive at
\begin{equation}\label{eq:spefPDE}
    \partial_t f=\nabla_\xx\cdot[D_T\nabla_\xx f-v_0\bf{e}_\theta f]+
    \partial_\theta[D_R\partial_\theta f-\gamma(\bf{n}_\theta\cdot\nabla c_\lambda)f],
\end{equation}
where $c_\lambda(t,\xx, \theta) = c(t, \xx+\lambda {\bf e}(\theta))$ and 
\begin{align}\label{def_rho}
\begin{aligned}
c(t,\xx) &= \int_\Sigma 	\nabla K_{\alpha,D}  (\xx -\xx_2 )f(t, \xi_2) \dd \xi_2 = \int_{\T^2_L} 	\nabla K_{\alpha,D}  (\xx  -\xx_2 )\rho(t, \xx_2) \dd \xx_2.
\end{aligned}
\end{align}

To obtain the propagation of chaos $f_2^N\rightharpoonup(f_1^\infty)^{\otimes 2}$ as $N\to\infty$ rigorously for this model one may consider the results for mean-field limits for 
Vicsek models (see for example \cite{bolley2012mean,briant2022cauchy} with global Lipschitz alignment kernel), Vlasov--Fokker--Planck models with singular interaction kernel \cite{bresch2022new,bresch2019mean}, or the Keller--Segel model \cite{10.1214/16-AAP1267}.
For a general review of the propagation of chaos, see also \cite{chaintron2021propagation}. 

%%%%%%%%%%%%%%%%%%%%%%%%%%%%%%%%
\section{Stationary solutions}\label{sec:FEM}
%%%%%%%%%%%%%%%%%%%%%%%%%%%%%%%%%
In this section, we consider the $y$-independent stationary solutions of the macroscopic model \eqref{eq:fcresc}, that is, $f(x,\theta)$ and $c(x)$ satisfying
\begin{subequations} \label{stat_2d}
    \begin{align}
    0&=\partial_x[D_T\partial_x f-\Pe\cos(\theta)f]+\partial_\theta[\partial_\theta f+\gamma\sin(\theta)\partial_x c(x+\lambda\cos(\theta))f],\\
    0&= \partial_x^2 c-\alpha c+\rho,
\end{align}
\end{subequations}
subject to periodic boundary conditions and mass and positivity constraints
$$\int f\dd x\dd\theta=1, \qquad f\geq 0.
$$
We use the following algorithm to approximate solutions to \eqref{stat_2d}
\begin{algorithmic}
    \STATE{$\text{res}_{\text{tot}} \gets 10^{10}$}
    \STATE{$c\gets c_0(x)=1-\frac{1}{2\pi}\cos(2\pi x)$}
    \WHILE{$\text{res}_{\text{tot}}>\text{tol}$}
    \STATE{$f\gets\text{FEM}_f(c)$}
    \STATE{$c\gets\text{FEM}_c(f)$}
    \STATE{$\text{res}_{\text{tot}}\gets\text{res}_{\text{tot}}(f,c)$}
    \ENDWHILE.
\end{algorithmic}
Here the functions $\text{FEM}_f,\text{FEM}_c$ are the finite element approximations of the equations for $f$ and $c$, respectively, and $\text{res}_\text{tot}$ is the maximum of the absolute values of the test residuals for all test functions used in the finite element method. We use the finite element method package \texttt{Fenics}.

Solutions of \eqref{stat_2d} for $\lambda=0$ and $\lambda=0.1$ using this algorithm are shown in \cref{fig:femlanes}. When $\lambda=0$, the solution is a ``one-dimensional spot'' at $x=0$, with particles oriented either left ($\theta = \pi$) or right ($\theta = 0$). Increasing the look-ahead parameter to $\lambda=0.1$ results in the shifting of the density to two peaks around $\theta = \pm \pi/2$, meaning that ants are mainly moving up and down a lane at $x = 0.5$. The tilting of the peaks is consistent with that of the solution obtained via time-dependent simulations (\cref{fig:fytheta_plot}) and the eigenfunction associated with the most unstable mode (\cref{fig:eigf_effect_lambda}).

\begin{figure}[htb]
    \centering
   \subfloat[$\lambda=0$]{\includegraphics[width=0.48\textwidth]{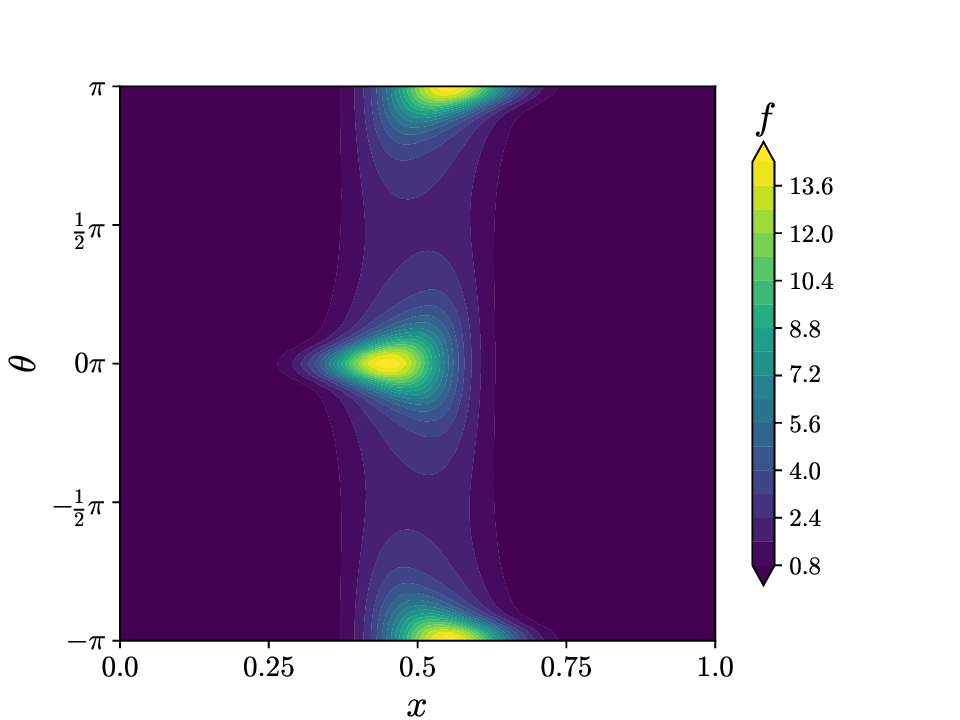}}
   \subfloat[$\lambda=0.1$]{\includegraphics[width=0.48\textwidth]{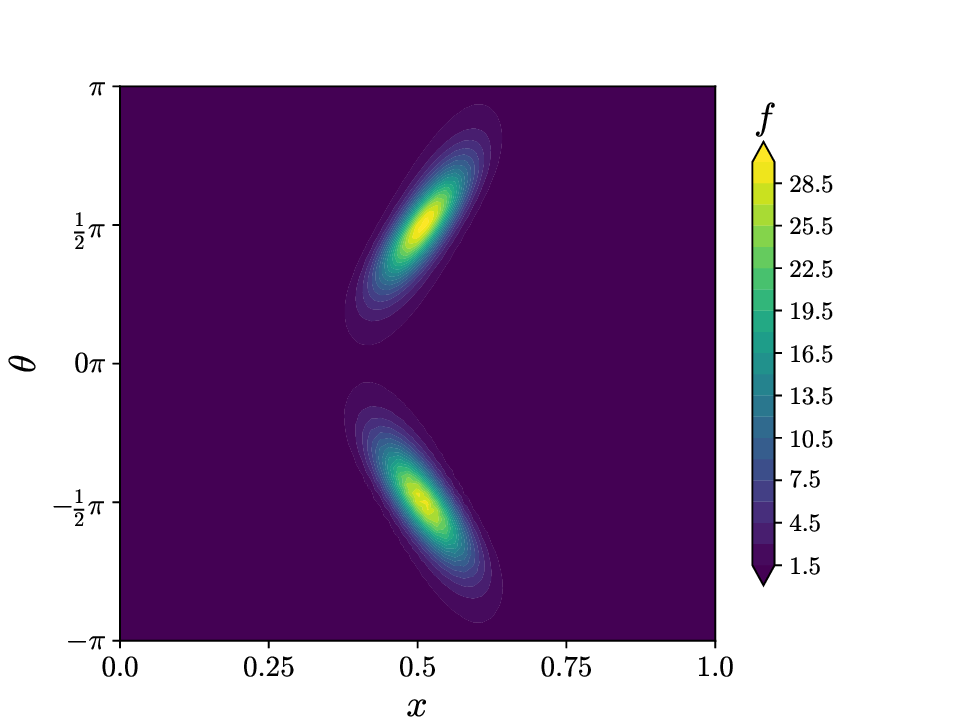}}
    \caption{Solutions $f(x,\theta)$ of \cref{stat_2d} for $\lambda = 0$ (a) and $\lambda=0.1$ (b) obtained via the finite-element method. Other parameters used: $\text{res}_{\text{tot}}=2.9\cdot 10^{-12}, D_T=0.1,\Pe=5.0,\gamma=50,\alpha=1$.}
    \label{fig:femlanes}
\end{figure}

\section{Sweep of spatial density plots with varying initial conditions, $\Pe$ and $\gamma$} \label{sec:summary_plots}
%%%%%%%%%%%%%%%%%%%%%%%%%%%

\cref{fig:rho706,fig:rho1001,fig:rho4472,fig:rho5555,fig:rho6061,fig:rho8154,fig:rho9437,fig:rho9956} show the solution $\rho(t,\xx)$ of \cref{eq:fcresc} at time $t=5.0$ for a range of parameter pairs $(\gamma, \Pe)$ and eight different initial conditions, obtained by sampling from the uniform distribution and then normalising, using different random seeds. This data is then used to generate \cref{fig:l2p2}.
\begin{figure}[htb]
\centering
\includegraphics[height=0.45\textheight]{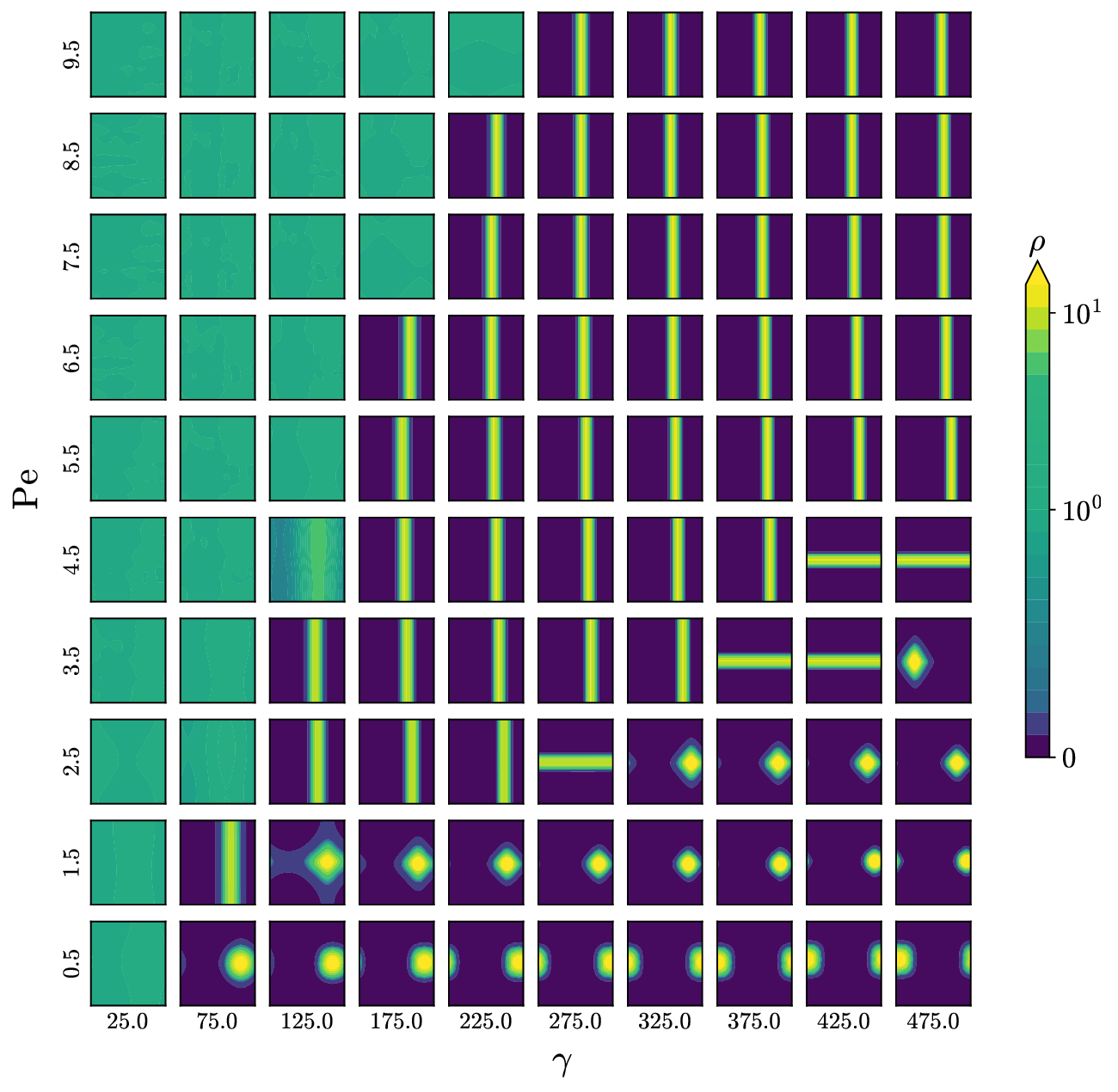}
\caption{Solution $\rho(t=5.0,\xx)$ of \cref{eq:fcresc} for a range of parameter values $\gamma$ and $\Pe$ and random seed 706. Other parameters are $D_T=0.01,\lambda=0.1,\alpha=1,N_x=N_y=31,N_\theta=21,\Delta t=10^{-5}$.}
\label{fig:rho706}
\end{figure}
\begin{figure}[htb]
\centering
\includegraphics[height=0.45\textheight]{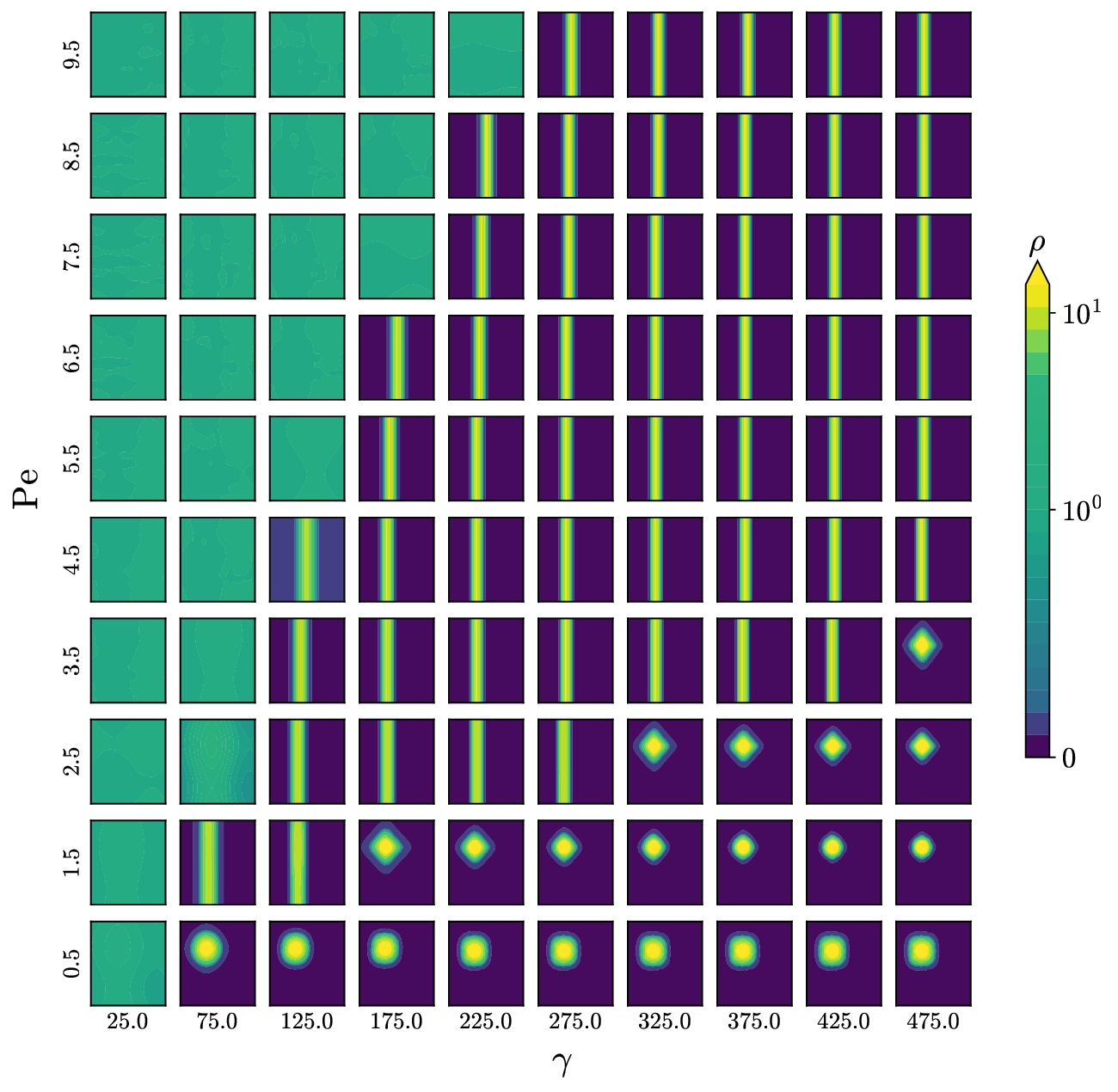}
\caption{Solution $\rho(t=5.0,\xx)$ of \cref{eq:fcresc} for a range of parameter values $\gamma$ and $\Pe$ and random seed 1001. Other parameters are $D_T=0.01,\lambda=0.1,\alpha=1,N_x=N_y=31,N_\theta=21,\Delta t=10^{-5}$.}
\label{fig:rho1001}
\end{figure}
\begin{figure}[htb]
\centering
\includegraphics[height=0.45\textheight]{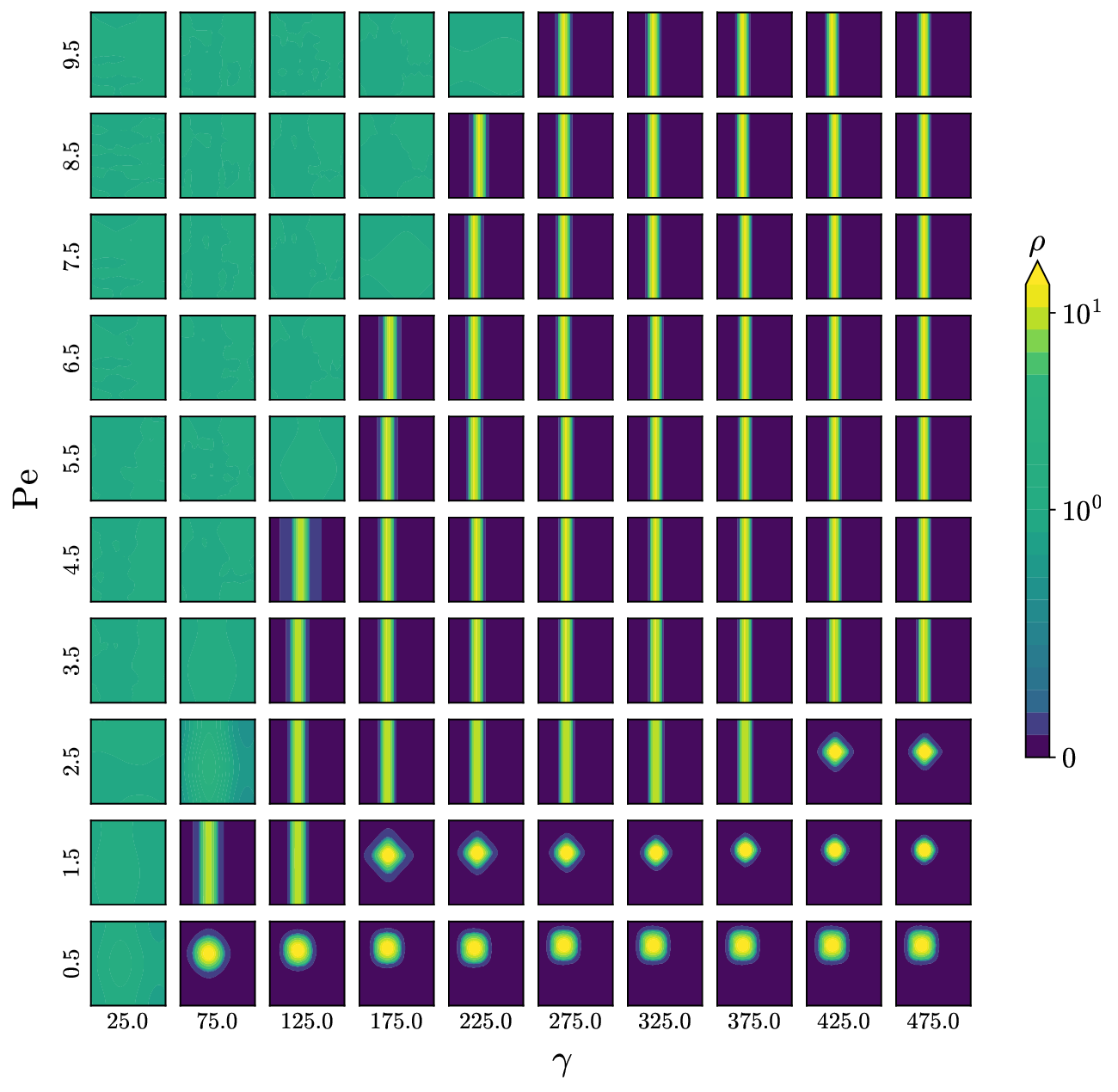}
\caption{Solution $\rho(t=5.0,\xx)$ of \cref{eq:fcresc} for a range of parameter values $\gamma$ and $\Pe$ and random seed 4472. Other parameters are $D_T=0.01,\lambda=0.1,\alpha=1,N_x=N_y=31,N_\theta=21,\Delta t=10^{-5}$.}
\label{fig:rho4472}
\end{figure}
\begin{figure}[htb]
\centering
\includegraphics[height=0.45\textheight]{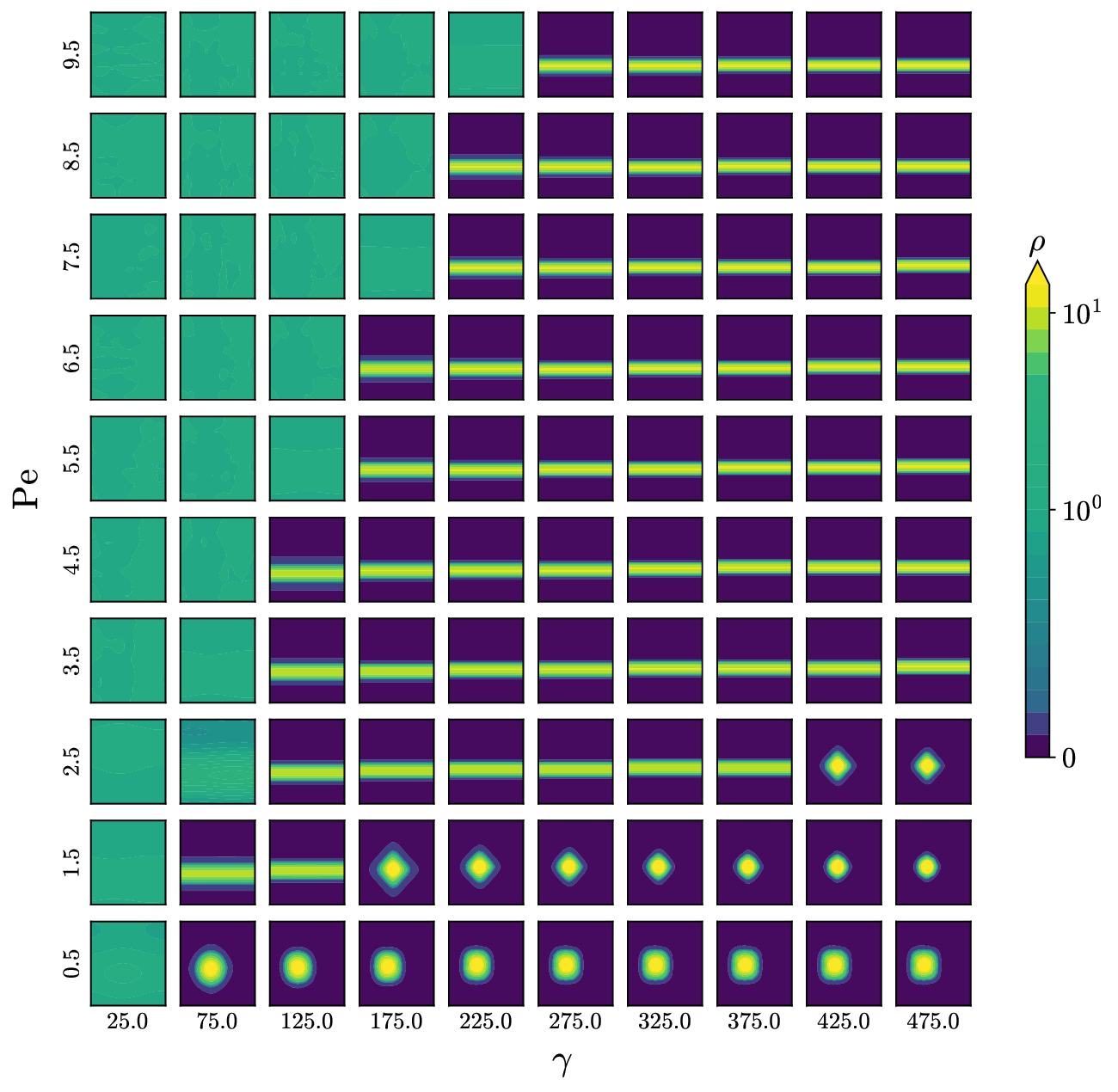}
\caption{Solution $\rho(t=5.0,\xx)$ of \cref{eq:fcresc} for a range of parameter values $\gamma$ and $\Pe$ and random seed 5555. Other parameters are $D_T=0.01,\lambda=0.1,\alpha=1,N_x=N_y=31,N_\theta=21,\Delta t=10^{-5}$.}
\label{fig:rho5555}
\end{figure}
\begin{figure}[htb]
\centering
\includegraphics[height=0.45\textheight]{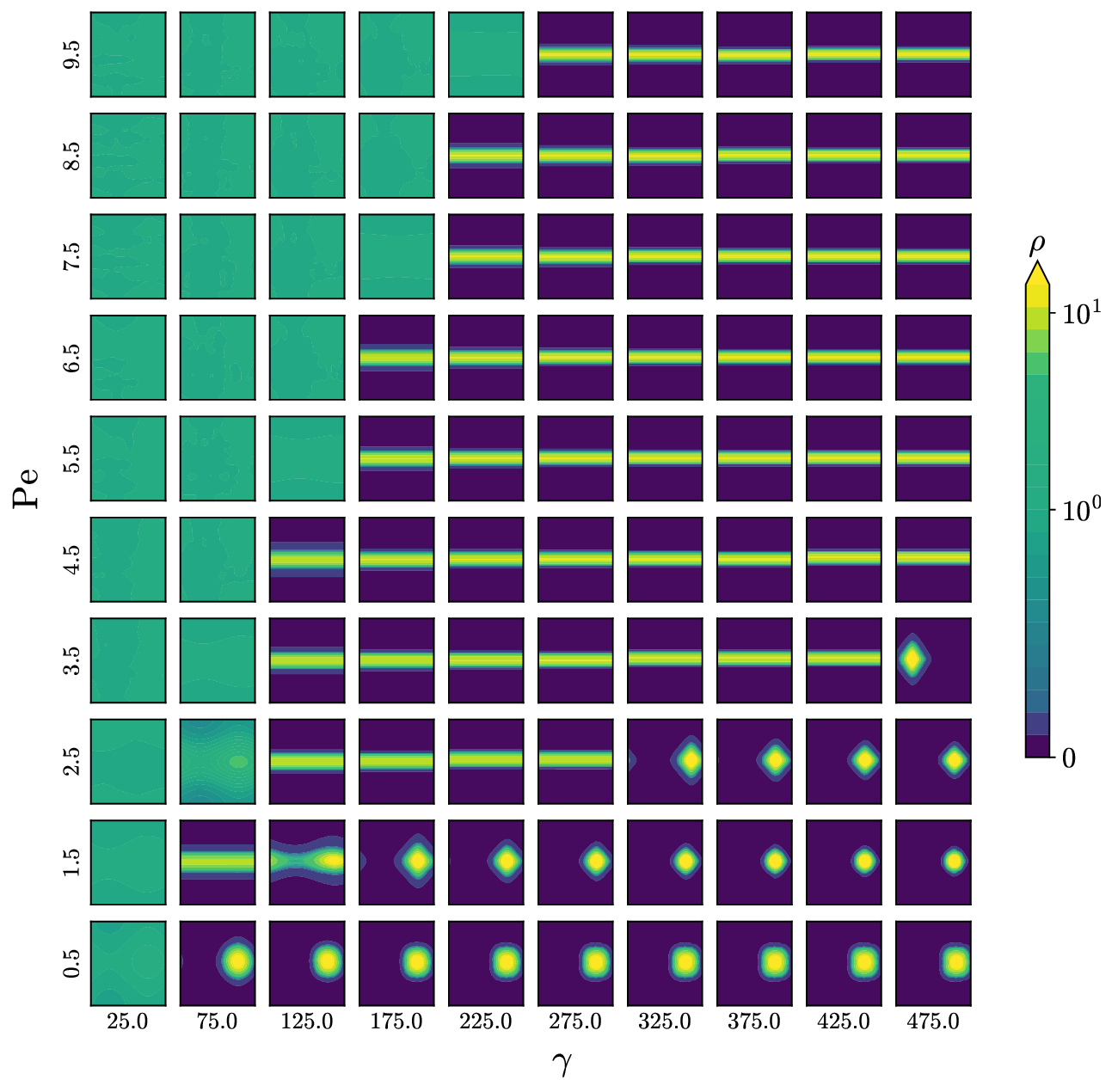}
\caption{Solution $\rho(t=5.0,\xx)$ of \cref{eq:fcresc} for a range of parameter values $\gamma$ and $\Pe$ and random seed 6061. Other parameters are $D_T=0.01,\lambda=0.1,\alpha=1,N_x=N_y=31,N_\theta=21,\Delta t=10^{-5}$.}
\label{fig:rho6061}
\end{figure}
\begin{figure}[htb]
\centering
\includegraphics[height=0.45\textheight]{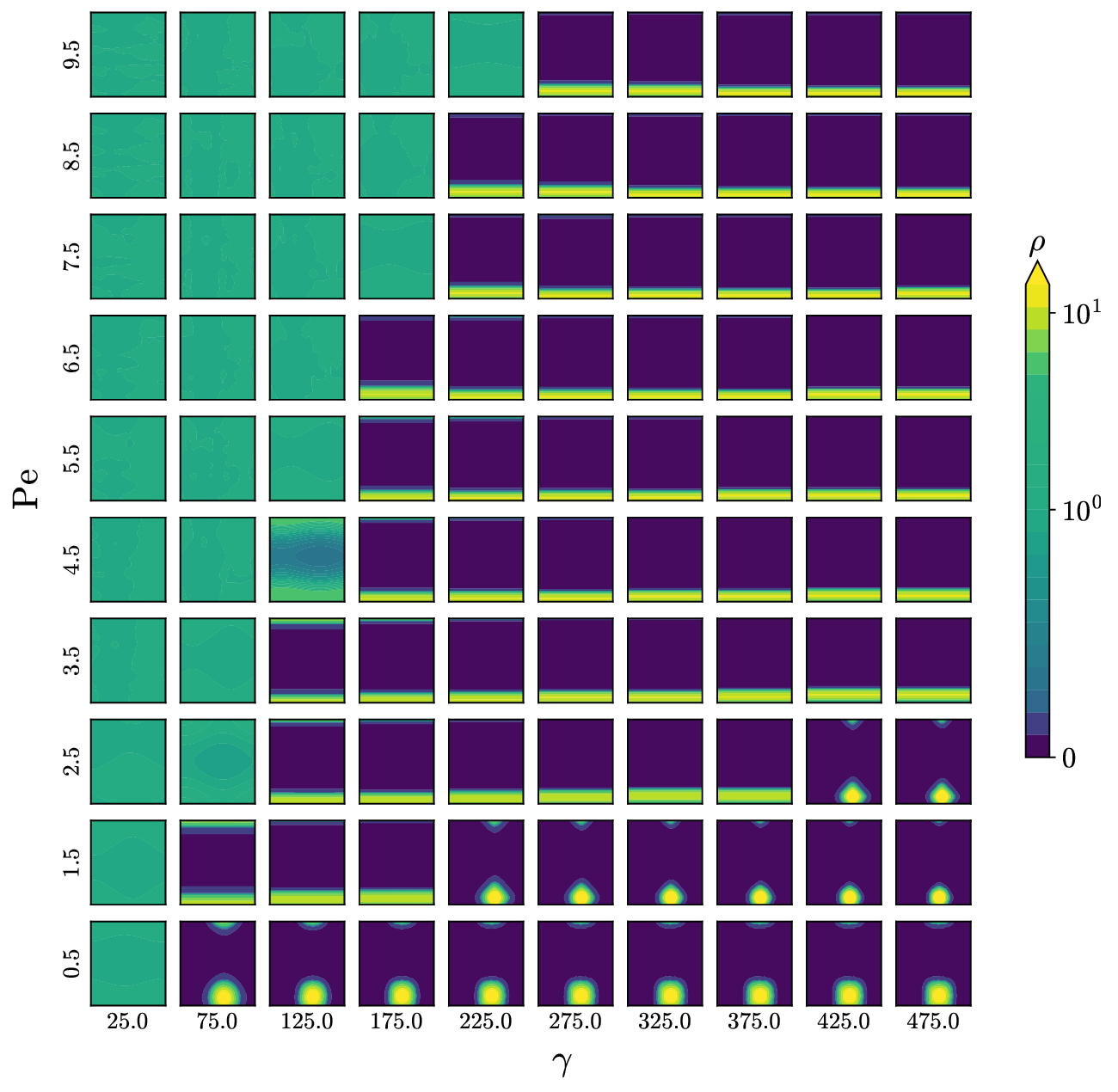}
\caption{Solution $\rho(t=5.0,\xx)$ of \cref{eq:fcresc} for a range of parameter values $\gamma$ and $\Pe$ and random seed 8154. Other parameters are $D_T=0.01,\lambda=0.1,\alpha=1,N_x=N_y=31,N_\theta=21,\Delta t=10^{-5}$.}
\label{fig:rho8154}
\end{figure}
\begin{figure}[htb]
\centering
\includegraphics[height=0.45\textheight]{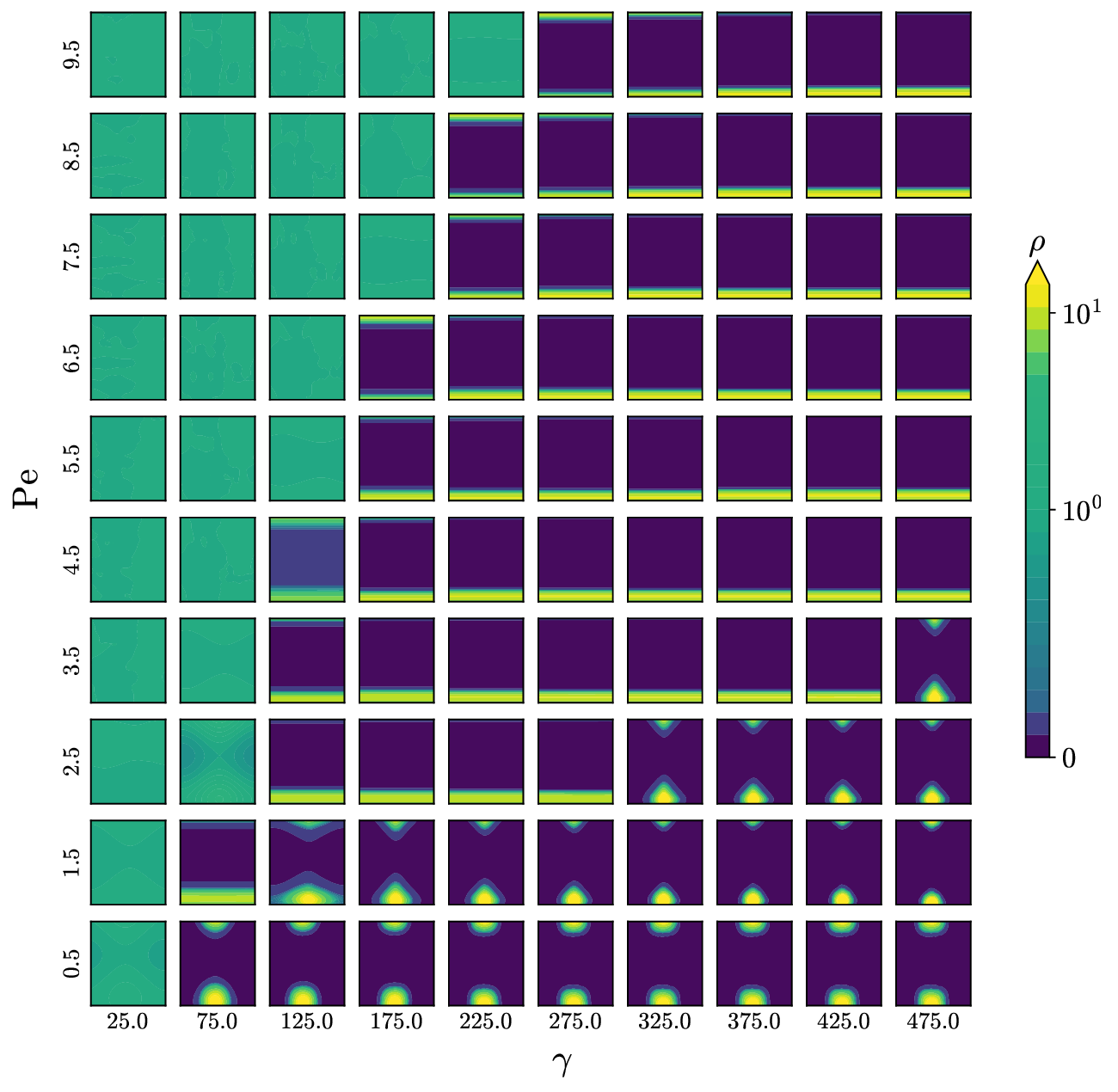}
\caption{Solution $\rho(t=5.0,\xx)$ of \cref{eq:fcresc} for a range of parameter values $\gamma$ and $\Pe$ and random seed 9437. Other parameters are $D_T=0.01,\lambda=0.1,\alpha=1,N_x=N_y=31,N_\theta=21,\Delta t=10^{-5}$.}
\label{fig:rho9437}
\end{figure}
\begin{figure}[htb]
\centering
\includegraphics[height=0.45\textheight]{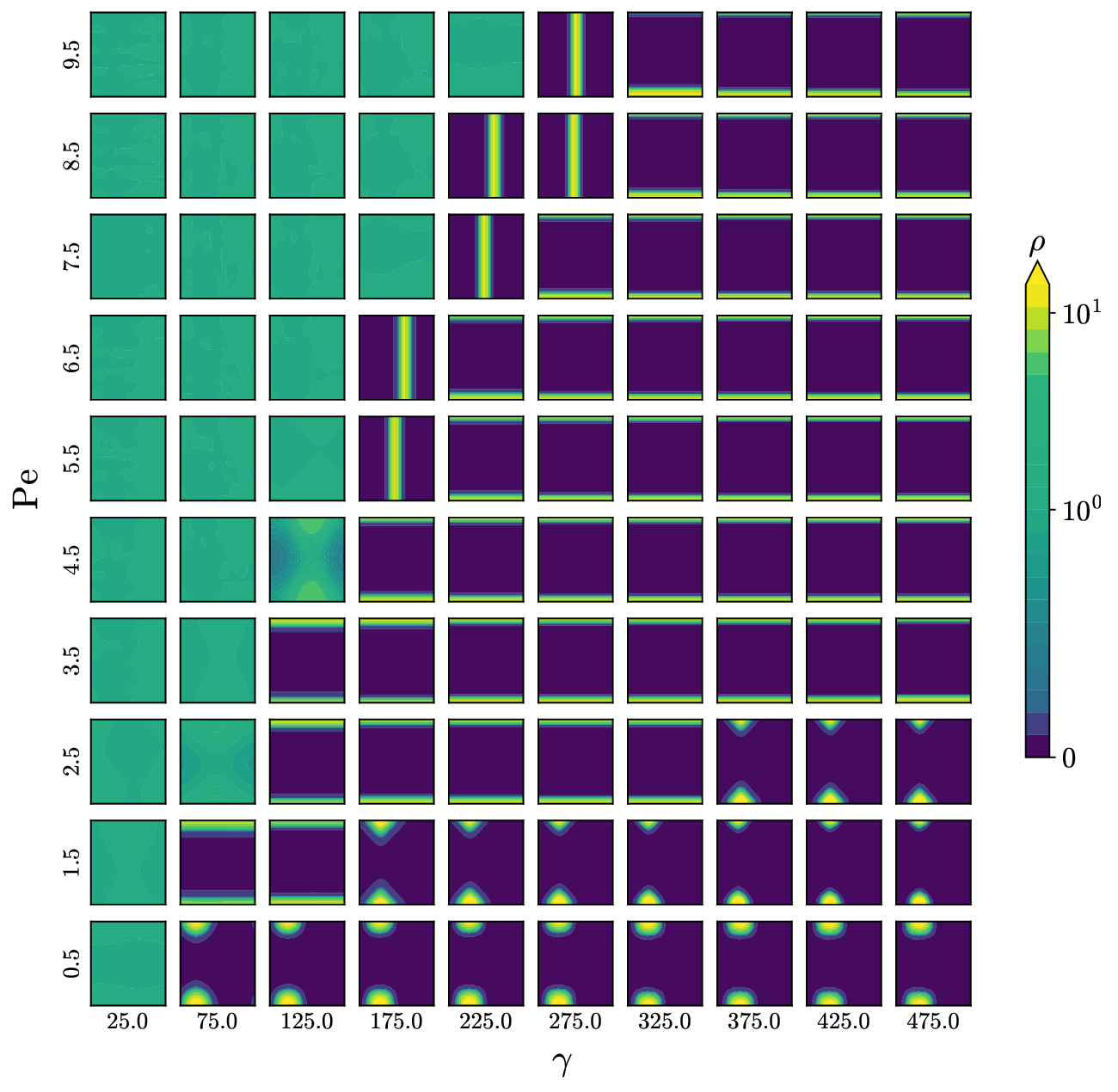}
\caption{Solution $\rho(t=5.0,\xx)$ of \cref{eq:fcresc} for a range of parameter values $\gamma$ and $\Pe$ and random seed 9956. Other parameters are $D_T=0.01,\lambda=0.1,\alpha=1,N_x=N_y=31,N_\theta=21,\Delta t=10^{-5}$.}
\label{fig:rho9956}
\end{figure}

\bibliographystyle{siamplain}
\bibliography{references}
\end{document}